\documentclass[11pt,a4paper]{amsart}
\usepackage{graphicx,latexsym,amsfonts,amsmath,amssymb,txfonts,enumerate,multirow}
\usepackage{epic}
\input xy
\xyoption{all}

\newcommand{\Hom}{ \,{\rm Hom} \,}

\newcommand{\im}{ \,{\rm Im} \,}

\theoremstyle{plain}
\newtheorem{theorem}{Theorem}[section]

\newtheorem{conjecture}[theorem]{Conjecture}
\newtheorem{lemma}[theorem]{Lemma}
\newtheorem{proposition}[theorem]{Proposition}
\theoremstyle{definition}
\newtheorem{definition}[theorem]{Definition}

\newtheorem{remark}[theorem]{Remark}

\newcommand{\RR}{{\mathbb R }}
\newcommand{\CC}{{\mathbb C }}

\newcommand{\ZZ}{{\mathbb Z }}
\newcommand{\PP}{ {\mathbb P }}
\newcommand{\tPP}{ {\tilde{\mathbb P }}}
\newcommand{\QQ}{{\mathbb Q }}

\newcommand{\GG}{{\mathbb G }}

\newcommand{\kk}{{\mathbf k }}

\newcommand{\env}{\!
\mathbin{\text{\rotatebox[origin=c]{70}{\scalebox{1.2}{$\approx$}}}} \!}

\newcommand{\weight}{\omega}

\newcommand{\newmmxu}{\mathfrak{m}_{\hat{U},X,\Omega}}
\newcommand{\newmmxh}{\mathfrak{m}_{H,X,\Omega}}
\newcommand{\newmmxg}{\mathfrak{m}_{G,X,\Omega}}
\newcommand{\newmm}{\mathfrak{m}_{G,X,\Omega}}

\newcommand{\calx}{\mathcal{X}}
\newcommand{\calo}{\mathcal{O}}

\newcommand{\cale}{\mathcal{E}}

\newcommand{\reg}{\mathrm{reg}}

\newcommand{\Diff}{\mathrm{Diff}}

\newcommand{\hU}{\hat{U}}

\newcommand{\res}{\operatornamewithlimits{Res}}

\newcommand{\coeff}{\mathrm{coeff}}

\newcommand{\grass}{\mathrm{Grass}}

\newcommand{\bi}{\mathbf{i}}

\newcommand{\be}{\mathbf{e}}

\newcommand{\jetk}[2]{J_{k}({#1},{#2})}

\newcommand{\jetreg}[2]{J_{k}^{\mathrm{reg}}({#1},{#2})}

\newcommand{\liek}{{\mathfrak k}}
\newcommand{\lieu}{{\mathfrak u}}

\newcommand{\lieks}{{\liek}^*}

\newcommand{\GL}{\mathrm{GL}}
\newcommand{\sym}{\mathrm{Sym}}

\newcommand{\lieg}{{\mathfrak g}}

\newcommand{\lieh}{\mathfrak{h}}

\newcommand{\diff}{\mathrm{Diff}}
\newcommand{\blow}{\mathrm{Bl}}
\newcommand{\DS}{\mathrm{DS}}
\newcommand{\tZmin}{\tilde{Z}_{\min}}
\newcommand{\GIT}{\mathrm{GIT}}

\DeclareMathOperator{\stab}{Stab}
\DeclareMathOperator{\Stab}{Stab}

\def\a{\alpha}
\def\b{\beta}

\def\d{\delta}

\def\l{\lambda}

\def\s{\sigma}

\title[Non-reductive GIT and hyperbolicity]{Non-reductive geometric invariant theory and hyperbolicity}

\author{Gergely B\'erczi}
\address{Department of Mathematics, Aarhus University} 
\email{gergely.berczi@math.au.dk}
\author{Frances Kirwan}
\address{Mathematical Institute, University of Oxford} 
\email{kirwan@maths.ox.ac.uk}

\begin{document}

\vspace{.5in}
\maketitle
\begin{center}
{\em In memoriam Jean-Pierre Demailly}
\end{center}

\begin{abstract} The Green--Griffiths--Lang and Kobayashi hyperbolicity conjectures for generic hypersurfaces of polynomial degree are proved using intersection theory for non-reductive geometric invariant theoretic quotients and recent work of Riedl and Yang.
\end{abstract} 

\maketitle

\section{Introduction}\label{sec:intro}

The goal of this paper is to apply a recent extension \cite{bdhk2,bkcoh} of geometric invariant theory (GIT)  to suitable non-reductive actions to study hyperbolicity of generic hypersurfaces in a projective space. We use the results of \cite{bdhk2} to construct new compactifications of bundles of invariant jet differentials over complex manifolds. Intersection theory developed in \cite{bkcoh} for non-reductive GIT quotients, combined with the strategy of \cite{dmr}, leads us to a proof of the Green--Griffiths--Lang (GGL) conjecture for a generic projective hypersurface whose degree is bounded below by a polynomial in its dimension. A recent result of Riedl and Yang \cite{riedl} then implies the polynomial Kobayashi conjecture. These results are significant improvements of the earlier known degree bounds, from $(\sqrt{n}\log n)^n$ to $16n^3(5n+4)$ in the case of the GGL conjecture and from $(n\log n)^n$ to $16(2n-1)^3(10n-1)$ in the case of the  Kobayashi conjecture.  

A projective variety $X$ is called Brody hyperbolic if there is no non-constant entire holomorphic curve in $X$, i.e. any holomorphic map $f: \CC \to X$ must be constant. Hyperbolic algebraic varieties have attracted considerable attention, in part because of their conjectured diophantine properties. For instance, Lang \cite{lang} has conjectured that any hyperbolic complex projective variety over a number field $K$ can contain only finitely many rational points over $K$. In 1970 Kobayashi \cite{kob2} formulated the following conjecture (more precisely, this is a slightly stronger version of Kobayashi's original conjecture in which \lq very general' replaced \lq generic'):

\begin{conjecture}[\textbf{Kobayashi conjecture, 1970}]
A generic hypersurface $X\subseteq \PP^{n+1}$ of degree $d_n$ is Brody hyperbolic if $d_n$ is sufficiently large.
\end{conjecture}

This conjecture has become a landmark in the field and has been the subject of intense study \cite{deng,brotbek,b1,b2,darondeau}. For more details on recent results see the survey papers \cite{demsurvey,dr}. Siu \cite{siu4} and Brotbek \cite{brotbek} proved the Kobayashi hyperbolicity of projective hypersurfaces of sufficiently high (but not effective) degree, and effective degree bounds were worked out by Deng \cite{deng} and Demailly \cite{demsurvey}. The conjectured optimal degree bound is $d_1=4, d_n=2n+1$ for $n=2,3,4$ and $d_n=2n$ for $n\ge 5$, see \cite{demsurvey}. The best known bound before the results of this paper was $d_n=(n \log n)^n$ by Merker and The-Anh Ta \cite{merker2}. 

A related, but stronger, conjecture is the Green--Griffiths--Lang (GGL) conjecture formulated in 1979 by Green and Griffiths \cite{gg} and in 1986 by Lang \cite{lang}. 
\begin{conjecture}[\textbf{Green-Griffiths-Lang conjecture, 1979}]
Any projective algebraic variety $X$ of general type contains a proper algebraic subvariety $Y\subsetneqq X$ such that every
nonconstant entire holomorphic curve $f:\CC \to X$ satisfies $f(\CC) \subseteq Y$. 
\end{conjecture}

In particular, a generic projective hypersurface $X\subseteq \PP^{n+1}$ is of general type if $\deg(X)\ge n+3$.    
A positive answer to the GGL conjecture has been given for surfaces by McQuillan \cite{mcquillan} under the assumption that the second Segre number $c^2_1-c_2$ is positive. Siu \cite{siu1,siu2,siu3,siu4} and Demailly \cite{dem} developed a powerful strategy to approach the conjecture for generic hypersurfaces $X\subseteq \PP^{n+1}$ of high degree. Following this strategy, combined with techniques of Demailly \cite{dem}, the first effective lower bound for the degree of a generic hypersurface in the GGL conjecture was given by Diverio, Merker and Rousseau \cite{dmr}, where the conjecture for generic projective hypersurfaces $X\subseteq \PP^{n+1}$ of degree $\deg(X)>2^{n^5}$ was confirmed. In 
 \cite{b1} the first author introduced equivariant localisation on the Demailly--Semple tower and adapted the argument of \cite{dmr} to improve this lower bound to $\deg(X)>n^{8n}$. The residue formula of \cite{b1} was later studied and further analyzed by Darondeau \cite{darondeau}. 
The best bound before the results of this paper for the Green--Griffiths--Lang Conjecture was $\deg(X)>(\sqrt{n}\log n)^n$, 
due to Merker and The-Anh Ta \cite{merker2} achieved by a deeper study of the formula of \cite{b1}. 
 
 In this paper we replace the Demailly--Semple bundle with a computationally more efficient algebraic model coming from non-reductive geometric invariant theory \cite{bdhk1} and apply the equivariant intersection theory developed in \cite{bkcoh} to prove 
 
\begin{theorem}[\textbf{Polynomial Green-Griffiths-Lang theorem for projective hypersurfaces}] \label{mainthmtwo}
Let $X\subseteq \PP^{n+1}$ be a generic smooth projective hypersurface
of degree $\deg(X)\ge 16n^3(5n+4)$. Then there is a proper algebraic subvariety $Y\subsetneqq X$ containing all nonconstant entire holomorphic curves in $X$. 
\end{theorem}  

Recently Riedl and Yang \cite{riedl} proved the following beautiful statement: if there are integers $d_n$ for all positive $n$ such that the GGL conjecture for generic hypersurfaces of dimension $n$ holds for degree at least $d_n$ then the Kobayashi conjecture is true for generic hypersurfaces with degree at least $d_{2n-1}$. Using this, Theorem \ref{mainthmtwo} immediately implies 

\begin{theorem}[\textbf{Polynomial Kobayashi theorem}] \label{mainthmtwob}
A generic smooth projective hypersurface $X\subseteq \PP^{n+1}$  
of degree $\deg(X)\ge 16(2n-1)^3(10n-1)$ is Brody hyperbolic. 
\end{theorem}

The strategy of Demailly and Siu is based on first establishing algebraic degeneracy of holomorphic curves $f:\CC \to X$, in the sense of proving the existence of certain polynomial differential equations of some order $k$, and finding enough such equations  $P(f',\ldots,f^{(k)})=0$  so that they cut out a proper algebraic locus $Y\subsetneqq X$. 
The central tool for finding polynomial differential equations  is the study of the bundle $J_kX$ of $k$-jets
of germs of holomorphic curves $f:\CC \to X$ over $X$, and the associated Green--Griffiths bundles  $E_{k}^{GG}=\oplus_{m=0}^\infty E_{k,m}^{GG}$ of algebraic differential operators where the fibres of $E_{k,m}^{GG}$ are polynomial functions $Q(f',\ldots,f^{(k)})$ of weighted degree $m$ in $f',\ldots,f^{(k)}$.  In \cite{dem} Demailly introduced the subbundles $E_{k,m} \subseteq E_{k,m}^{GG}$ consisting of jet differentials which are (semi-)invariant under reparametrisation of the source $\CC$. The group $\Diff_k(1)$ of $k$-jets of reparametrisation germs $(\CC,0) \to (\CC,0)$ at the origin acts fibrewise on $J_kX$, and $\oplus_{m=1}^\infty E_{k,m}$ is the graded algebra of  jet differentials which are invariant under the maximal unipotent subgroup $U_k$ of $\Diff_k(1)$. This bundle gives a better reflection of the geometry of entire curves in $X$, since
it only depends on the images of such curves and not on their
parametrisations. However, it also comes with a technical difficulty: the reparametrisation group $\Diff_k(1)$ is not reductive, and so the classical geometric invariant theory of Mumford \cite{git} cannot be applied to study the invariants and construct a compactification of a quotient of a nonempty open subset $J_k^{\mathrm{reg}}X$ of $J_kX$ by $\Diff_k(1)$ (cf. 
 \cite{dk}). Until recently, there existed only two different constructions for the compactification of these quotients. 
\begin{enumerate}
\item In  \cite{dem} Demailly describes a  smooth compactification of $J_k^{\mathrm{reg}}X/\Diff_k(1)$ as a tower of projectivised bundles on $X$ --- the Demailly--Semple bundle --- endowed with line bundles such that sections of their powers give $\Diff_k(1)$-(semi)-invariants. Global sections of properly chosen twisted tautological line bundles over the Demailly--Semple bundle give algebraic differential equations of degree $k$. This model was extensively and successfully used in the past few decades, and it has a vast literature in hyperbolicity questions (see also \cite{demsurvey,dr}). The main numerical breakthrough in the Green-Griffiths-Lang conjecture using the Demailly-Semple tower was achieved in \cite{dmr}, where the first effective bound for the degree of a generic projective hypersurface was calculated. However, as  was pointed out in \cite{b1}, we cannot expect better than an exponential bound in the GGL conjecture using the Demailly-Semple model. 
\item In \cite{b2} the first author shows that the curvilinear component of the punctual Hilbert scheme of $k$ points on $\CC^n$ provides natural compactifications of the fibres of $J_k^{\mathrm{reg}}X/\Diff_k(1)$ over $X$. Sections of the tautological bundle give invariant jet differentials, and equivariant localisation developed in \cite{bsz,bsz2,bercziG&T,berczitau2,berczitau3,berczitau4} gives information on the intersection theory of this curvilinear component. In \cite{b2} it is shown that the GGL conjecture for generic hypersurfaces with polynomial degree follows from a classical positivity conjecture of Rim\'anyi \cite{rimanyi} for Thom polynomials. However, Rim\'anyi's conjecture is currently out of reach.
\ 
\end{enumerate}

The first key idea of this paper is to replace these existing models with a  new construction, coming from a recently developed extension of geometric invariant theory (GIT) to suitable  non-reductive group actions \cite{bdhk2}. We construct a projective  completion $\calx_k^{\GIT}$ of $\calx_k^{\mathrm{reg}}       = {J}_k^{\mathrm{reg}}
X/\Diff_k(1)$, which fibres over $X$ with fibres given by non-reductive GIT quotients. 
 The projective completion $\calx_k^{\GIT}$ is
 endowed with a relatively ample line bundle such that sections of its powers give (semi-)invariant jet differentials. 
 From this we construct a tautological line bundle over $\calx_k^{\GIT}$, and
 following Diverio--Merker--Rousseau \cite{dmr} (with the improved pole order obtained by Darondeau \cite{darondeau2} for slanted vector fields) and using the holomorphic Morse inequalities \cite{dem00,trap}, we show that the existence of nonzero sections of this line bundle follows from the positivity of a certain tautological integral over $\calx_k^{\GIT}$. 

The second key ingredient of this paper is the cohomological intersection theory for non-reductive GIT quotients developed in \cite{bkcoh}, which allows us to prove the positivity of this integral at the critical order $k=n$ for hypersurfaces with polynomial degree. 

The layout of the paper is as follows. Background material needed on jet differentials and on non-reductive GIT is recalled in $\S$2 and $\S$3. In $\S$4 we summarise the results of \cite{bkcoh} on moment maps for nonreductive quotients and their cohomological applications. $\S$5 contains the construction of the 
 projective completion $\calx_k^{\GIT} $ of ${J}_k^{\mathrm{reg}}
X/\Diff_k(1)$ fibring over $X$, and describes the fibres $X_k^{\GIT}$ of $\calx_k^{\GIT} \to X$ as non-reductive GIT quotients, with explicit descriptions when $k=2,3$. In $\S$6 it is shown that the proof of Theorem \ref{mainthmtwo} reduces to proving that a tautological integral 
$$ \int_{\calx_n^{\GIT}} I_{n,\delta} $$
is strictly positive for suitable $\delta>0$ where $n = \dim X$. In $\S$7 the cohomology of the fibres $X_k^{\GIT}$ is studied, and $\S$8 derives integration formulas on $\calx_k^{\GIT}$. In $\S$9 the integral $ \int_{\calx_n^{\GIT}} I_{n,\delta} $ is calculated explicitly when $n=2$, and 
then the proof of Theorem \ref{mainthmtwo} is completed by demonstrating positivity for general $n$.

The authors would like to thank the anonymous referee for helpful comments on an earlier version of this paper.

\section{Jet differentials}\label{sec:jetdiff}

The central object of this paper is the algebra of (semi-)invariant jet differentials under
reparametrisation of the source space $\CC$. For more details see the survey papers of Demailly \cite{dem} and Diverio--Rousseau \cite{dr}.

\subsection{Jets of holomorphic maps}\label{subsubsec:holmaps} 

If $u,v$ are positive integers let $J_k(u,v)$ denote the vector
space of $k$-jets of holomorphic maps $(\CC^u,0) \to (\CC^v,0)$ at
the origin; that is, the set of equivalence classes of maps
$f:(\CC^u,0) \to (\CC^v,0)$, where $f\sim g$ if and only if
$f^{(j)}(0)=g^{(j)}(0)$ for all $j=1,\ldots ,k$. This is a finite-dimensional complex
vector space, which can be identified with $J_k(u,1) \otimes \CC^v$; hence
$\dim J_k(u,v) =v \binom{u+k}{k}-v$.  We will call the elements of
$J_k(u,v)$ {\em map-jets of order} $k$, or simply map-jets. 

Eliminating the terms of degree $k+1$ results in a surjective algebra homomorphism
$J_k(u,1) \twoheadrightarrow J_{k-1}(u,1)$, and the sequence of such surjections
$J_k(u,1) \twoheadrightarrow J_{k-1}(u,1)
\twoheadrightarrow \ldots \twoheadrightarrow J_1(u,1)$ induces
an increasing filtration of $J_k(u,1)^*$:
\begin{equation}\label{filtration} 
J_1(u,1)^* \subseteq J_2(u,1)^* \subseteq \ldots \subseteq J_k(u,1)^*.
\end{equation}

Using the standard coordinates on $\CC^u$ and $\CC^v$, a $k$-jet $f \in J_k(u,v)$ can be identified with its collection of derivatives at the
origin, the vector $(f'(0),f''(0),\ldots, f^{(k)}(0))$, where
$f^{(j)}(0)\in \mathrm{Hom}(\mathrm{Sym}^j\CC^u,\CC^v)$. Here $\sym^l$ denotes the symmetric tensor product. In this way we
get an isomorphism 
\begin{equation}\label{identification}
J_k(u,v) \simeq J_k(u,1) \otimes \CC^v \simeq \oplus_{j=1}^k\mathrm{Hom}(\mathrm{Sym}^j\CC^u,\CC^v).
\end{equation}

Map-jets can be composed via substitution and elimination of terms
of degree greater than $k$, leading to the composition map
\begin{equation}
  \label{comp0}
\jetk uv \times \jetk vw \to \jetk uw,\;\;  (\Psi_1,\Psi_2)\mapsto
\Psi_2\circ\Psi_1 \mbox{ modulo terms of degree $>k$ }.
\end{equation}
When $k=1$, we can identify $J_1(u,v)$  with the space of $u$-by-$v$ matrices,
and \eqref{comp0} reduces to multiplication of matrices.

We will call a jet $\gamma\in J_k(u,v)$ {\em regular} if
$\gamma'(0)$ is has maximal rank, and we will use the notation $\jetreg uv$ for the set
of regular maps.
When $u=v$ we get a group 
\[\mathrm{Diff}_k(u)=\jetreg uu\]
which we will call the {\em $k$-jet diffeomorphism group}. 

\begin{remark} \label{commut} Note that from (\ref{comp0}) we obtain commuting right and left actions of $\mathrm{Diff}_k(1)$ and $\mathrm{Diff}_k(n)$ on $J_k(1,n)$. \end{remark}

\subsection{Jet bundles and differential operators}\label{subsec:jetbundles}

Let $X$ be a smooth projective variety of dimension $n$. Following Green and Griffiths \cite{gg} we let $J_kX \to X$ be the bundle of $k$-jets of germs of parametrised curves in $X$; that is, $J_kX$ is the of equivalence classes of germs of holomorphic
maps $f:(\CC,0) \to (X,p)$, where the equivalence relation $\sim$ is such that $f\sim g$
if and only if the derivatives $f^{(j)}(0)$ and $g^{(j)}(0)$ are equal for
$0\le j \le k$ when computed in some local holomorphic coordinate system on an open neighbourhood 
 of $p\in X$. The projection map $J_kX \to X$ is given by $f \mapsto f(0)$, 
and the elements of the fibre $J_kX_p$ can be represented by Taylor expansions 
\[f(t)=p+tf'(0)+\frac{t^2}{2!}f''(0)+\ldots +\frac{t^k}{k!}f^{(k)}(0)+O(t^{k+1}) \]
 up to order $k$ at $t=0$ of $\CC^n$-valued maps $f=(f_1,f_2,\ldots, f_n)$
on open neighbourhoods of 0 in $\CC$. Locally in these coordinates elements of the fibre $J_kX_p$ can be identified with $k$-tuples of vectors $(f'(0),\ldots, f^{(k)}(0)/k!) \in (\CC^n)^k$, so the fibre can be identified with $J_k(1,n)$.

Note that $J_kX$ is not
a vector bundle over $X$ since the transition functions are polynomial but
not linear, see \S 5 of Demailly \cite{dem}. In fact, let $\diff_X$ denote the principal $\diff_k(n)$-bundle over $X$ formed by all local polynomial coordinate systems on $X$. Then 
we have an identification
\[J_kX \cong \diff_X \times_{\diff_k(n)} J_k(1,n)\] 
of $J_kX$ with the associated affine bundle whose structure group is $\diff_k(n)$, acting on $J_k(1,n)$ as in Remark \ref{commut}. With respect to this identification, $\diff_k(1)$ acts on $J_kX$ fibrewise via its action on $J_k(1,n)$ as in Remark \ref{commut}.

Let $J_k^{\reg}X$ denote the bundle of $k$-jets of germs of parametrised curves $f:\CC \to X$ in $X$ which are regular in the sense that they have nonzero first derivative $f'(0)\neq 0$. After fixing local coordinates near $p\in X$, the fibre $J_k^{\reg}X_p$ can be identified with $\jetreg 1n$ and 
\[J_k^{\reg}X \cong \diff_X \times_{\diff_k(n)} J_k^\reg(1,n).\]


\subsection{Invariant jet differentials}\label{subsec:jetdiff}
Let $X$ be a complex $n$-dimensional manifold and let $k$ be a positive integer. Recall that after choosing local coordinates on $X$ near $p$ we can identify the fibre $J^\reg_kX_p$ of $J^\reg_k X$ at $p$ with $\jetreg 1n$. We can explicitly write out the reparametrisation action (defined at \eqref{comp0}) of $\diff_k(1)$ on $\jetreg 1n$ as follows. Let $z f'(0)+\frac{z^2}{2!}f''(0)+\ldots +\frac{z^k}{k!}f^{(k)}(0) \in \jetreg 1n$ be the $k$-jet of a germ at the origin (i.e. no constant term) in $\CC^n$ with $f^{(i)}(0)\in \CC^n$  for $1 \leq i \leq k$  such that $f'(0) \neq 0$, and let  
$\varphi(z)=\alpha_1z+\alpha_2z^2+\ldots +\alpha_k z^k \in \jetreg 11$ with $\alpha_i\in \CC, \alpha_1\neq 0$.  
 Then the $k$-jet of 
 $f \circ\varphi(z)$ is 
\[ (f'(0)\alpha_1)z+(f'(0)\alpha_2+
\frac{f''(0)}{2!}\alpha_1^2)z^2+\ldots
+\left(\sum_{i_1+\ldots +i_l=k}
\frac{f^{(l)}(0)}{l!}\alpha_{i_1}\ldots \alpha_{i_l}\right)z^k\]
\begin{equation}\label{jetdiffmatrix}
=(f'(0),\ldots, f^{(k)}(0)/k!)\cdot 
\left(
\begin{array}{ccccc}
\alpha_1 & \alpha_2   & \alpha_3          & \ldots & \alpha_k \\
0        & \alpha_1^2 & 2\alpha_1\alpha_2 & \ldots & 2\alpha_1\alpha_{k-1}+\ldots \\
0        & 0          & \alpha_1^3        & \ldots & 3\alpha_1^2\alpha_ {k-2}+ \ldots \\
0        & 0          & 0                 & \ldots & \cdot \\
\cdot    & \cdot   & \cdot    & \ldots & \alpha_1^k
\end{array}
 \right)
 \end{equation}
where the $(i,j)$ entry  of this matrix is $p_{i,j}(\bar{\alpha})=\sum_{a_1+a_2+\ldots +a_i=j}\alpha_{a_1}\alpha_{a_2} \ldots \alpha_{a_i}.$

\begin{remark}\label{naturalembedding}
The linear representation of $\diff_k(1)$ on $\jetreg 1n$ given by \eqref{jetdiffmatrix} embeds $\diff_k(1)$ as a upper triangular subgroup of $\GL(n)$. Thus $\diff_k(1)$ is a linear algebraic group but is not reductive for $k \geq 2$, so Mumford's classical GIT cannot be used to construct compactifications of the orbit space $\jetreg 1n/ \diff_k(1)$ (cf. \cite{dk,bdhk2}). 

This matrix group is parametrised along its first row with free parameters $\alpha_1\in \CC^*, \alpha_2,\ldots, \alpha_k \in \CC$, while the other entries are certain (weighted homogeneous) polynomials in these free parameters. It is a semidirect product \[\diff_k(1)=U_k \rtimes \CC^*\]
of its unipotent radical $U_k$ by a one-parameter subgroup $\CC^*$ acting diagonally. 
Here $U_k$ is the subgroup given by substituting $\alpha_1=1$, and the diagonal subgroup $\CC^*$ acts with strictly positive weights $1\ldots, n-1$ on the Lie algebra $\mathrm{Lie}(U_k)$ of $U_k$. In \cite{bdhk2,bkcoh}  actions of non-reductive groups of this type  are studied in a more general context. 
\end{remark}

The action
of  $\l \in \CC^*$ on $k$-jets is thus described by
\[\l\cdot (f'(0),f''(0),\ldots ,f^{(k)}(0))=(\l f'(0),\l^2 f''(0),\ldots,
\l ^kf^{(k)}(0)).\]
Let $\cale_{k,m}$ denote the space of complex-valued polynomials $Q(f'(0),f''(0),\ldots ,f^{(k)}(0))$ on $J_k(1,n)$ of weighted degree $m$ with respect to this $\CC^*$ action; that is, they satisfy
\[Q(\l f'(0),\l^2 f''(0),\ldots, \l^k f^{(k)}(0))=\l^m Q(f(0)',f''(0),\ldots,
f^{(k)}(0)).\]
In \cite{dem} Demailly introduced the Green-Griffiths bundle $E_{k,m}^{GG}$ over $X$ whose fibres are $\cale_{k,m}$, observing that  the concept of polynomial on the fibres of $J_kX$ using local coordinates on $X$ is well defined. 
From our viewpoint this
 can be written as the associated bundle 
\begin{equation}\label{GGassociated}
E_{k.m}^{GG}=\Diff_X \times_{\diff_k(n)} \cale_{k,m}. 
\end{equation}
The Green-Griffiths bundle of order $k$ is then $E_k^{GG}=\oplus_{m\ge 0} E_{k,m}^{GG}$.

The fibrewise $\diff_k(1)$ action on $J_kX$ induces an action on $E_{k,m}^{GG}$. Demailly in \cite{dem} defined the  bundle of invariant jet differentials of order $k$ and weighted degree $m$ as the subbundle $E_{k,m}\subseteq E_{k,m}^{GG}$ of polynomial differential operators $Q(f'(0),\ldots, f^{(k)}(0))$ which are invariant under $U_k$; that is for any $\varphi \in  \diff_k(1)$ 
\[Q((f\circ \varphi)'(0),(f\circ \varphi)''(0), \ldots, (f\circ
\varphi)^{(k)}(0))=\varphi'(0)^m\cdot Q(f'(0),f''(0),\ldots, f^{(k)}(0)).\]
We call $E_k=\oplus_m E_{k,m}=(\oplus_mE_{k,m}^{GG})^{U_k}$ the Demailly--Semple bundle of invariant jet differentials.

\subsection{Compactifications of $J_kX/\diff_k(1)$}

In order to find and describe invariant jet differentials we can try  to construct projective completions of the quasi-projective fibrewise quotient 
\[J_k^\reg X/\diff_k(1)=\diff_X \times_{\diff_k(n)} (J_k^\reg(1,n)/\diff_k(1)).\]
This quotient fibres over $X$ (as $\diff_k(1)$ acts fibrewise) and we can hope to detect invariant jet differentials as global sections of powers of relatively ample line bundles on suitable  fibrewise projective completions $\overline{J_k^\reg X/\diff_k(1)}$. 
Indeed this strategy works, and there exist two constructions in the literature.

\begin{enumerate} \item  \textbf{The Demailly--Semple tower} The first construction goes back to Semple, and was studied and introduced into the study of hyperbolicity questions by the landmark paper of Demailly \cite{dem}. The Demailly--Semple tower $X_k^\DS$ is an iterated projective bundle over $X$ 
\[X_k^{\DS} \to X_{k-1}^\DS \to \ldots \to X_1^\DS \to X_0^\DS=X\]
endowed with projections $\pi_{i,k}:X_k^\DS \to X_i^\DS$ and canonical line bundles $\pi_{i,k}^*\calo_{X_i^\DS}(1) \to X_k^\DS$ whose sections are global invariant jet differentials. The total space $X_k$ is smooth of dimension $\dim(X_k)=n+k(n-1)$. For the details of the construction see \cite{dem,dmr}.
In \cite{b1} equivariant localisation was introduced on the Demailly-Semple tower, and following the strategy of \cite{dmr}, the Green--Griffiths--Lang conjecture for generic hypersurfaces with degree at least $n^{6n}$ was proved. In \cite{b1} it was also proved that we cannot expect better than an exponential degree bound with this approach. 

\item \textbf{The curvilinear component of the Hilbert scheme of $k$ points on $\CC^n$} In \cite{bercziG&T} the first author proves that the curvilinear component $\mathrm{CHilb}^{k+1}_0(\CC^n)$ of the punctual Hilbert scheme $\mathrm{Hilb}^k_0(\CC^n)$ supported at the origin is a compactification of the fibre $J_k^\reg(1,n)/\diff_k(1)$ of $J_k^\reg X/\diff_k(1)$. More precisely, in \cite{bercziG&T} (following \cite{bsz}) we describe a $\diff_k(1)$-invariant  map
\begin{equation}\label{embedlocal}
\phi: J_k^\reg(1,n) \to \grass(k,J_k(n,1)^*).
\end{equation}
and we identify the image in the Grassmannian with the punctual curvilinear locus of the Hilbert scheme of $k+1$ points in $\CC^n$ supported at the origin. The curvilinear locus consists of  subschemes of length $k+1$ on $\CC^n$ which are supported at the origin and are limits of $k+1$ distinct points colliding along a smooth curve:
\begin{equation}\label{curvilinear}
\im(\phi)=\{I \subset \CC[x_1,\ldots, x_n]: \CC[x_1,\ldots, x_n]/I \simeq \CC[t]/t^k\} \subset \mathrm{Hilb}^{k+1}_0(\CC^n)
\end{equation}
The closure of the curvilinear locus is a component of the punctual Hilbert scheme, the so-called curvilinear component $\mathrm{CHilb}^{k+1}_0(\CC^n)$. Since $\phi$ is injective on the $\diff_k(1)$-orbits,
\[\mathrm{CHilb}^{k+1}_0(\CC^n)=\overline{\im(\phi)}=\overline{J_k^\reg(1,n)/\diff_k(1)}\]
is a projective completion of $J_k^\reg(1,n)/\diff_k(1)$. By moving the support on $X$ we get the fiberwise projective completion 
\[\mathrm{CHilb}^{k+1}(X)=\overline{J_k^\reg X/\diff_k(1)}\] 
Using equivariant localisation, in \cite{b2} we connect hyperbolicity of hypersurfaces with global singularity theory and Thom polynomials of $A_n$-singularities. Modulo a positivity conjecture of Rim\'anyi, \cite{b2} obtains the Green--Griffiths--Lang conjecture for hypersurfaces with polynomial ($>n^8$) degree. However, a complete proof of the positivity conjecture currently seems to be out of reach. 
\end{enumerate} 
This paper introduces a third compactification coming from the recent development of a version of GIT which applies to suitable non-reductive group actions \cite{bdhk2}. As we have seen, the reparametrisation group $\diff_k(1)$ is not reductive, but it is a linear algebraic group with internally graded unipotent radical in the sense of \cite{bdhk2}, and hence the construction and results of this paper apply. We use the fibrewise completion 
\[  
\calx_k^{\GIT} :=\diff_X \times_{\diff_k(n)} (\tPP[\CC \oplus J_k(1,n)]/\!/\diff_k(1))\]
where $\tPP[\CC \oplus J_k(1,n)]/\!/\diff_k(1)
$ is a non-reductive GIT quotient of a projective space blown up at the point $[1:0: \cdots :0]$. 
Care is needed in this construction, since $J_kX$ is an affine bundle but not a vector bundle over $X$, and the fibrewise action of $\diff_k(n)$ on $J_kX$ does not extend to the projective space $\PP[\CC \oplus \Hom(\CC^k,\CC^n)]$ nor to its blow-up $\tPP[\CC \oplus \Hom(\CC^k,\CC^n)]$. However we will see that the projective non-reductive GIT quotient $\tPP[\CC \oplus \Hom(\CC^k,\CC^n)]/\!/ \diff_k(1)$ is a geometric quotient of a $\diff_k(n)$-invariant open subset of the blow-up of $J_k(1,n)$ at 0, and thus $\diff_k(n)$ does act naturally on $\tPP[\CC \oplus \Hom(\CC^k,\CC^n)]/\!/\diff_k(1)$ by Remark \ref{commut}. Moreover the affine bundle $J_kX$ over $X$ has a well-defined zero section which is invariant under the fibrewise action of $\diff_k(n)$, and we can identify $\calx_k^{\GIT}$ with the geometric quotient of the action of $\diff_k(1)$ on an open subset of the blow-up of $J_kX$ along its zero section.

This means that we can apply the intersection theory and integration formulas for non-reductive GIT quotients proved in \cite{bkcoh} as our main computational tool, to integrate over the fibres of $\calx_k^{\GIT} \to X$. 

\begin{remark} The Demailly-Semple tower $X_k^\DS$, the curvilinear Hilbert scheme $\mathrm{CHilb}^{k+1}(X)$ and the non-reductive GIT quotient $\calx_k^{\GIT}$ are three different fiberwise compactifications of the invariant jet differentials bundle $ J_k^\reg X/\diff_k(1)$. Description of their birationality is, however, a difficult problem. There are rational maps 
\[\xymatrix{\calx_k^{\GIT} \ar@{-->}[r]^-\phi & \mathrm{CHilb}^{k+1}(X) & X_k^\DS \ar@{-->}[l]_-\psi }\]
where $\phi$ is defined in \eqref{embedlocal}, and $\psi$ comes from choosing local coordinates on the Demailly-Semple tower, see \cite{berczibrotbek}.
\begin{enumerate}[(i)]
\item In \cite{berczithom} a toric blow-up of $\phi$ is defined: we obtain a morphism $\hat{\phi}: \hat{\calx}_k^{\GIT} \to \mathrm{CHilb}^{k+1}(X)$ and we pull-back integration over $\mathrm{CHilb}^{k+1}(X)$ to the non-reductive GIT quotient $ \hat{\calx}_k^{\GIT}=\widehat{J_kX}/\diff_k(1)$ to arrive new toric residue formulas for Thom polynomials of singularities. 
\item In the unpublished note \cite{berczibrotbek} we study $\psi$, and the blow-up of the Demailly-Semple bundle at the Wronskian ideal sheaf which defines the indeterminacy locus of $\psi$. The main goal of \cite{berczibrotbek} was the study of the disribution of torus fixed points (i.e monomial ideals) on the components of the punctual Hilbert scheme. This problem was later solved in \cite{berczisvendsen}.
\end{enumerate}

\end{remark}

\section{Non-reductive geometric invariant theory}\label{sec:nrgit}

In \cite{bdhk2} an extension of Mumford's classical GIT is developed for linear actions of a non-reductive linear algebraic group with internally graded unipotent radical  over an algebraically closed field $\kk$ of characteristic 0; the results are summarised and slightly extended in \cite{bkcoh}. In this section we will follow the definitions of \cite{bkcoh} and quote its results.

\begin{definition}\label{def:gradedunipotent} 
We say that a linear algebraic group $H = U \rtimes R$  has {\em internally graded unipotent radical} $U$ if there is a central one-parameter subgroup $\l:\GG_m \to Z(R)$ of the Levi subgroup $R$ of $H$ 
 such that the adjoint action of $\GG_m$ on the Lie algebra of $U$ has all its weights strictly positive. 
 Then $\hat{U} = U \rtimes \l(\GG_m)$ is a normal subgroup of $H$ and $H/\hU \cong R/\l(\GG_m)$ is reductive.
\end{definition}

Let $H=U \rtimes R$ be a linear algebraic group with internally graded unipotent radical $U$ acting linearly with respect to an ample line bundle $L$ on a projective variety  $X$; that is, the action of $H$ on $X$ lifts to an action on $L$ via automorphisms of the line bundle. When $H=R$ is reductive, using Mumford's classical geometric invariant theory (GIT)  \cite{git}, we can define $H$-invariant open subsets $X^s \subseteq X^{ss}$ of $X$ (the stable and semistable loci for the linearisation) with a geometric quotient $X^s/H$ and (if $X^s$ is non-empty)  projective completion $X/\!/H \supseteq X^s/H$ which is the projective variety associated to the algebra  of invariants
 $\bigoplus_{k \geq 0} H^{0}(X,L^{\otimes k})^H$. 
Here the variety $X/\!/H$ is the image
of a surjective morphism $\phi$ from the open subset $X^{ss}$ of $X$  such that if $x,y \in X^{ss}$ then $\phi(x) = \phi(y)$ if and only if the closures of the $H$-orbits of $x$ and $y$ meet in $X^{ss}$. 
Furthermore the subsets $X^s$ and $X^{ss}$ can be described using the Hilbert--Mumford criteria for stability and semistability. 

Mumford's GIT 
 does not have an immediate extension  to actions of non-reductive linear algebraic groups $H$, since the algebra of invariants $\bigoplus_{k \geq 0} H^{0}(X,L^{\otimes k})^H$ is not necessarily finitely generated as a graded algebra when $H$ is not reductive. 
It is still possible to define semistable and stable subsets $X^{ss}$ and $X^s$, with a geometric quotient $X^s/H$ which is an open subset of a so-called enveloping quotient $X\env H$ with an $H$-invariant morphism $\phi: X^{ss} \to X\env H$, and if the algebra of invariants 
 $\bigoplus_{k \geq 0} H^{0}(X,L^{\otimes k})^H$ is finitely generated then $X\env H$ is the associated projective variety
 \cite{bdhk2,dk}. But in general the enveloping quotient $X\env H$ is not necessarily projective, the morphism $\phi$ is not necessarily surjective (and its image may be only a constructible subset, not a subvariety, of $X\env H$). In addition there are in general no obvious analogues of the Hilbert--Mumford criteria.

However when $H = U \rtimes R$ has internally graded unipotent radical $U$ and acts linearly on a projective variety $X$, then provided that we are willing to modify the linearisation of the action by replacing the line bundle $L$ by a sufficiently divisible tensor power and multiplying by a suitable character of $H$ (which will not change the action of $H$ on $X$), many of the  key features of classical GIT still apply.

Let such an $H$ act linearly on an irreducible projective variety $X$ with respect to a very ample line bundle $L$.
Let $\chi: H \to \GG_m$ be a character of $H$. Its kernel contains $U$, and its restriction to $\hU$ can be identified with an integer so that the integer 1 corresponds to the character of $\hU$ which fits into the exact sequence $U \hookrightarrow \hat{U} \to \l(\GG_m)$. Let $\weight_{\min}$ be the minimal weight for the $\l(\GG_m)$-action on
$V:=H^0(X,L)^*$ and let $V_{\min}$ be the weight space of weight $\weight_{\min}$ in
$V$. Suppose that $\weight_{\min}=\weight_0 < \weight_{1} < 
\cdots < \weight_{\max} $ are the weights with which the one-parameter subgroup $\l: \GG_m \leq \hU \leq H$ acts on the fibres of the tautological line bundle $\calo_{\PP((H^0(X,L)^*)}(-1)$ over points of the connected components of the fixed point set $\PP((H^0(X,L)^*)^{\GG_m}$ for the action of $\GG_m$ on $\PP((H^0(X,L)^*)$; since $L$ is very ample $X$ embeds in $\PP((H^0(X,L)^*)$ and the line bundle $L$ extends to the dual $\calo_{\PP((H^0(X,L)^*)}(1)$ of the tautological line bundle on $\PP((H^0(X,L)^*)$.
Note that we can assume that there exist at least two distinct such weights since otherwise the action of the unipotent radical $U$ of $H$ on $X$ is trivial, and so the action of $H$ is via an action of the reductive group $R=H/U$.

 The linearisation of the action of $H$ on $X$ with respect to the ample line bundle $L^{\otimes c}$ can be twisted by the character $\chi$ so that the weights $\weight_j$ are replaced with $\weight_jc-\chi$;
let $L_\chi^{\otimes c}$ denote this twisted linearisation. 

\begin{definition}\label{def:welladapted}
If 
$$ \weight_{\min} <
\frac{\chi|_{\hU} }{c} < \weight_{\min} + \epsilon,$$
where $\epsilon>0$ is sufficiently small, we will call the rational character $\chi/c$  {\it well adapted } to the linear action of $\hU$. 
We will call a linearisation {\it well adapted} if $\weight_{\min} <0<\weight_{\min} + \epsilon$ for sufficiently small $\epsilon>0$. Note that $\chi/c$ is well-adapted if and only if the linearisation $L_\chi^{\otimes c}$ is well-adapted, and every linearisation $L$ with at least two different $\l(\GG_m)$-weights has well-adapted rational twists. 

\end{definition}

\begin{remark} How small $\epsilon >0$ should be depends on the situation; see \cite{bkcoh} Definition 4.19 and the discussion following it.
We certainly require that $\epsilon \le \weight_1 - \weight_{\min}$ so that 
$(\chi |_{\hU})/c \in (\weight_{\min},\weight_1)$, but we may well require $\epsilon$ to be smaller than this (cf. the proof of Theorem 4.28 in \cite{bkcoh}). The main requirement is that $\epsilon>0$ should be sufficiently small to ensure via the theory of variation of GIT that, under suitable additional hypotheses (see Definition \ref{def:welladaptedaction} below), twisting the linearisation of the $H$-action on $X$ by the rational character $\chi/c$ results in (semi)stable loci for reductive subgroups of $H$ which are independent of $\chi/c$ subject to the requirement that $\weight_{\min} < (\chi |_{\hU})/c < \weight_{\min} + \epsilon$. By the Hilbert--Mumford criteria, this can be expressed in terms of the geometry of convex hulls of sets of weights for the action on $X$ of a maximal torus of $R$.

\end{remark}
Suppose that $H = U \rtimes R$, with grading one-parameter subgroup $\l:\GG_m \to Z(R)$, acts linearly on a projective variety $X$ with respect to an ample line bundle $L$.  Let
\[
Z_{\min}:=X \cap \PP(V_{\min})=\left\{
\begin{array}{c|c}
\multirow{2}{*}{$x \in X$} & \text{$x$ is a $\l(\GG_m)$-fixed point and} \\ 
 & \text{$\l(\GG_m)$ acts on $L^*|_x$ with weight $\weight_{\min}$} 
\end{array}
\right\}
\]
and
\begin{equation}\label{def:xmin0}
X_{\min}^0:=\{x\in X \mid p(x)  \in Z_{\min}\}  \quad \mbox{ where } \quad  p(x) =  \lim_{\substack{ t \to 0\\ t \in \GG_m }} \l(t) \cdot x \quad \mbox{ for } x \in X.
\end{equation}  
Then  $p: X_{\min}^0 \to Z_{\min}(X)$ is $U$-invariant and $R$-equivariant (\cite{bkcoh} Lemma 4.16); that is, 
\[p(rux)=rp(x) \text{ for } x \in X_{\min}^0, r \in R, u\in U.\]  

Suppose that $L_\chi^{\otimes c}$ is well adapted and let $X^{s,\GG_m}_{\min+}$ denote the stable subset of $X$ for the linear action of $\l:\GG_m \to Z(R)$ with respect to $L_\chi^{\otimes c}$. By the Hilbert--Mumford criteria (\cite{git,Newstead})
\[X^{s,\GG_m}_{\min+} = X_{\min}^0 \setminus Z_{\min};\] 
indeed $X^{s,\GG_m}_{\min+}$ is the stable set for the action of $\l:\GG_m \to Z(R)$ with respect to any $L_\chi^{\otimes c}$ such that 
$\weight_{\min} < \chi/c < \weight_{1}$.
Since the infinitesimal action of a weight vector $\xi \in \lieu$ with weight $\chi$ for the adjoint action of $\l(\GG_m)$ on the Lie algebra $\lieu$ of $U$ takes a weight vector in $V$ with weight $\weight$ to one of weight $\weight + \chi$, and takes a weight vector in the dual of $V$ with weight $\weight$ to one of weight $\weight - \chi$, the action of $U$ on $X$ preserves $X_{\min}^0$, although it  does not in general preserve $Z_{\min}$. Thus if $u \in U$ we have 
\begin{equation} \label{xsminplus} uX^{s,\GG_m}_{\min+} = X_{\min}^0 \setminus uZ_{\min}. \end{equation}

\begin{definition} \label{cond star}  
We say that 
 \lq semistability coincides with stability' for a well adapted linear action of $\hU$ 
  if 
\begin{equation}\label{star}  
 \stab_{U}(z) = \{ e \}  \textrm{ for every } z \in Z_{\min}. \tag{$*$}
\end{equation}
Note that \eqref{star} holds if and only if we have $\stab_{U}(x) = \{e\}$ for all $x \in X^0_{\min}$. 

\label{def min stable}   When (\ref{star}) holds for a well adapted action of $\hU$, the {\it min-stable locus} for the $\hU$-action is 
\[ X^{s,{\hU}}_{\min+}= 
\bigcap_{u \in U} u X^{s,\lambda(\GG_m)}_{\min+} = X^0_{\min} \setminus U Z_{\min}.  \]
Here the last equation follows from \eqref{xsminplus}.
\end{definition}

\begin{definition}\label{def:welladaptedaction} A {\it well-adapted linear action} 
on an irreducible projective variety $X$ is given by the following data: 
\begin{enumerate}
\item a linear algebraic group $H = U \rtimes R$ with unipotent radical $U$ and Levi subgroup $R$, and a central one-parameter subgroup $\l:\GG_m \to Z(R)$ of $R$ which grades $U$ in the sense of Definition \ref{def:gradedunipotent}, with a \lq complementary' connected subgroup $Z^\perp$ of $Z(R)$ such that $\mathrm{Lie}Z(R) = \mathrm{Lie}Z^\perp \oplus \mathrm{Lie}\lambda(\GG_m)$;
\item a linear action of $H$ on $X$ with respect to an ample line bundle $L$: 
\item a character $\chi: H \to \GG_m$ of $H$ whose restriction to $Z^\perp$ is zero and whose restriction to $\hU = U \rtimes \lambda(\GG_m)$ is nonzero, and  a strictly positive integer $c$ such that the rational character $\chi/c$
satisfies
$$(\chi |_{\hU})/c = \weight_{\min} + \epsilon \mbox{  where }0< \epsilon <\!< 1.$$
\end{enumerate}
One way to choose $Z^\perp$ is so that its Lie algebra is the orthogonal complement to $\mathrm{Lie}\lambda(\GG_m)$ in $\mathrm{Lie}Z(R)$ with respect to a suitable inner product. $Z^{\perp}$ determines and is determined by the ray in the direction of $\chi$ (that is, consisting of positive scalar multiples of $\chi$); 
 the rational character $\chi/c$ is required to lie on 
 the line generated by this ray, just beyond the 
  point labelled by $\weight_{\min}$.
\end{definition}

\begin{theorem}[\cite{bkcoh} Theorem 8] \label{mainthm} 
Let $(X,L,H,\hat{U},\chi)$ be a well-adapted linear action satisfying condition (\ref{star}) in Definition \ref{cond star}. Then 
\begin{enumerate}
\item the algebras of invariants 
\[\oplus_{m=0}^\infty H^0(X,L_{m\chi}^{\otimes cm})^{\hat{U}} \text{ and } \oplus_{m=0}^\infty H^0(X,L_{m\chi}^{\otimes cm})^{H} = (\oplus_{m=0}^\infty H^0(X,L_{m\chi}^{\otimes cm})^{\hat{U}})^{R}\]
are finitely generated;   
\item the enveloping quotient $X\env \hat{U}$ is the projective variety associated to the algebra of invariants $\oplus_{m=0}^\infty H^0(X,L_{m\chi}^{\otimes cm})^{\hat{U}}$ and is a geometric quotient of the open subset $X^{s,\hU}_{\min+}$ of $X$ by $\hat{U}$;
  

\item the enveloping quotient $X\env H$ is the projective variety associated to the algebra of invariants $\oplus_{m=0}^\infty H^0(X,L_{m\chi}^{\otimes cm})^{{H}}$ and is the classical GIT quotient of $X \env \hat{U}$ by the induced action of $R/\l(\GG_m)$ with respect to the linearisation induced by a sufficiently divisible tensor power of $L$. 
\end{enumerate}
\end{theorem}

\begin{remark} \label{explainN}
The reason that it is only a sufficiently divisible tensor power of $L$ which induces a line bundle on the geometric quotient $X^{s,\hU}_{\min+}/\hat{U}$ is that the action of $\hU$ on $X^{s,\hU}_{\min+}$ may have nontrivial (but finite) stabilisers. Replacing $L$ by $L^{\otimes N}$ where $N$ is sufficiently divisible ensures that $\Stab_{\hU}(x)$ acts trivially on the fibre of $L$ over $x$ for each $x \in X^{s,\hU}_{\min+}$.
\end{remark}

\begin{definition}
When the conditions of Theorem \ref{mainthm} hold, we call $X \env H$ (respectively $X \env \hU$) a {\it GIT quotient} and we denote it by $X/\!/H$ (respectively $X /\!/ \hU$). 
\end{definition}

\begin{definition}\label{def:s=ss} Let $(X,L,H,\hat{U},\chi,c)$ be ingredients for a well-adapted linear action (as in Definition \ref{def:welladaptedaction}) such that \eqref{star} holds (as in Definition \ref{cond star}). 
We denote by $X^{s,{{H}}}_{\min+}$ and $X^{ss,{{H}}}_{\min+}$ the pre-images in $X^{s,{{\hU}}}_{\min+} 
$ of the stable and semistable loci for the induced linear action of the reductive group $H/\hU = R/\l(\GG_m)$ on 
$X/\!/ \hat{U} = X^{s,\hU}_{\min+}/\hat{U}$.
We denote by $Z_{\min}(X)^{s(s),R/\l(\GG_m)}$  the (semi)stable locus for the action of the reductive group $R/\l(\GG_m)$ on $Z_{\min}(X)$, with the (rational) linearisation induced by twisting the original linearisation of the $R$-action by a rational multiple of $\chi$ chosen so that after twisting $\l(\GG_m)$ acts trivially on the restriction of $L$ to $Z_{\min}(X)$ (equivalently the twisted linearisation is borderline adapted in the sense of Definition \ref{def:welladapted}). Then we say that 
\begin{enumerate}[(i)] 
\item \emph{weak $H$-stability=$H$-semistability} holds if 
$X^{s,{{H}}}_{\min+} = X^{ss,{{H}}}_{\min+}$, and 
\item \emph{strong $H$-stability=$H$-semistability} holds if
$$Z_{\min}(X)^{s,R/\l(\GG_m)} = Z_{\min}(X)^{ss,R/\l(\GG_m)} \neq \emptyset.$$
\end{enumerate}
\end{definition}

\begin{remark}
These versions both coincide with condition \eqref{star} when $H=\hat{U}$, which will be the case in the situation considered in this paper. 
In general the weak version can depend on the choice of the character $\chi$, whereas the strong version is independent of this choice since $\l(\GG_m)$ acts trivially on $Z_{\min}(X)$ and on the restriction of $L$ to $Z_{\min}(X)$. 
\end{remark}

\begin{theorem}[\cite{bkcoh} Thm 4.28] 
Let $(X,L,H,\hat{U},\chi,c)$ be ingredients for a well-adapted linear action (as in Definition \ref{def:welladaptedaction}) such that semistability coincides with stability for the action of $H$  in the strong sense of Definition \ref{def:s=ss}. 
Suppose also that $UZ_{\min}(X)$ is not dense in $X$. 
Then  the projective variety $X/\!/ H$ is a geometric quotient of $X^{s,{{H}}}_{\min+}$
 by the action of $H$, the stabiliser  $\Stab_H(x)$ is finite for all $x \in X^{s,{{H}}}_{\min+}$
and 
$$ 
X^{s,H}_{\min+}  = X^{ss,{{H}}}_{\min+} = p^{-1}(Z_{\min}(X)^{s,R/\l(\GG_m)}) \setminus UZ_{\min}(X),$$
where $p:X^0_{\min} \to Z_{\min}(X)$ is  as in \eqref{def:xmin0} and $Z_{\min}(X)^{s,R/\l(\GG_m)}$ is  as 
 in Definition \ref{def:s=ss}. 
\end{theorem}

\begin{remark} \label{lineN}
Here by Theorem \ref{mainthm} $X/\!/ H$ is the projective variety associated to the finitely generated graded algebra 
$\oplus_{m=0}^\infty H^0(X,L_{m\chi}^{\otimes cm})^{H} $. Thus 
$X/\!/ H$ has an ample line bundle which we will denote by $\calo_{X/\!/ H}(1)$ whose pullback to 
$p^{-1}(Z_{\min}(X)^{s,R/\l(\GG_m)}) \setminus UZ_{\min}(X)$ 
is isomorphic to $L^{\otimes N}$ for some $N \geq 1$, and such that when $N$ divides $cm$ then $H$-invariant sections of $L_{m\chi}^{\otimes cm}$ define sections of $\calo_{X/\!/H}(cm/N)$. More precisely, sections of $\calo_{X/\!/H}(cm/N)$ pull back to $H$-invariant sections of the restriction of $L_{m\chi}^{\otimes cm}$ to $X^{s,H}_{\min+}$, and these all extend to $H$-invariant sections of 
$L_{m\chi}^{\otimes cm}$ over $X$; moreover every $H$-invariant section of 
$L_{m\chi}^{\otimes cm}$ over $X$ vanishes on the complement of $X^{s,H}_{\min+}$.

Note that although $H$ acts on $p^{-1}(Z_{\min}(X)^{s,R/\l(\GG_m)}) \setminus UZ_{\min}(X)$ with finite stabilisers it may not act freely, and if $x \in p^{-1}(Z_{\min}(X)^{s,R/\l(\GG_m)}) \setminus UZ_{\min}(X)$ then $\mathrm{Stab}_H(x)$ may not act trivially on the fibre $L_x$ of the line bundle $L$. However if $N$ is chosen to be sufficiently divisible then 
$\mathrm{Stab}_H(x)$ will act trivially on $L_x^{\otimes N}$ for all such $x$, so that $L^{\otimes N}$ descends to a line bundle on $X/\!/H$.

\end{remark}

\begin{remark} 
If  $UZ_{\min}(X)$ is dense in $X$, then $Z_{\min}(X)$ is the natural candidate for the role of quotient of $X$ by $U$ or $\hU$, with semistable locus $X^0_{\min}$, via the map 
$p:X^0_{\min} \to Z_{\min}(X)$ defined at \eqref{def:xmin0}, and $Z_{\min}(X)/\!/(R/\l(\GG_m))$ is the natural candidate for the quotient of $X$ by $H$. Thus in this case we are reduced to reductive GIT.
\end{remark}

In the set-up of the Green--Griffiths--Lang conjecture, the group which acts is a polynomial reparametrisation group of the form $\hU=U \rtimes \CC^*$, but initially the condition (\ref{star}) that semistability coincides with stability in Definition \ref{cond star} is not satisfied. However we can blow up to obtain a situation in which semistability does coincide with stability and then apply Theorem \ref{mainthm} to construct a non-reductive quotient; we can also apply Theorems \ref{thm:cohomologyrings}-\ref{jeffreykirwannonred} below in order to integrate over this non-reductive quotient. In fact this blow-up construction works much more generally; see \cite{bdhk1}.

\section{Moment maps and cohomology of non-reductive quotients}\label{NRGITlocalisation}

In this section we briefly summarise the results of \cite{bkcoh}, which generalise results of 
 Martin \cite{SM} to the cohomology of GIT quotients by non-reductive groups with internally graded unipotent radicals.

\subsection{Reductive actions} First let us recall the reductive picture. Let $X$ be a nonsingular complex projective variety acted on by a complex reductive group $G$ with respect to an ample linearisation. Then we can choose a maximal compact subgroup $K$ of $G$ and a $K$-invariant Fubini--Study K\"ahler metric on $X$ with corresponding moment map $\mu:X\to \mathfrak{k}^*$, where $\mathfrak{k}$ is the Lie algebra of $K$ and $\mathfrak{k}^* = \Hom_{\RR}(\mathfrak{k},\RR)$ is its dual. $\mathfrak{k}^*$ embeds naturally in the complex dual $\mathfrak{g}^* = \Hom_{\CC}(\mathfrak{k},\CC)$ of the 
Lie algebra $\mathfrak{g} = \mathfrak{k} \otimes \CC$ of $G$, as $\mathfrak{k}^* = \{\xi \in \mathfrak{g}^*: \xi(\mathfrak{k}) \subseteq \RR\}$; using this identification we can regard $\mu: X \to \mathfrak{g}^*$ 
 as a \lq moment map'  for the action of $G$, although of course it is not a moment map for $G$ in the traditional sense of symplectic geometry. 

In \cite{francesthesis} it is shown that the norm-square $f=||\mu||^2$ of the moment map $\mu:X\to \mathfrak{k}^*\subseteq \mathfrak{g}^*$ induces an equivariantly perfect Morse stratification of $X$ such that  the open stratum which retracts equivariantly onto the zero level set $\mu^{-1}(0)$ of the moment map coincides with the GIT semistable locus $X^{ss}$ for the linear action of $G$ on $X$. In particular this tells us that the restriction map
$$H^*_G(X;\QQ) \to H^*_G(X^{ss};\QQ)$$
is surjective; we also have an isomorphism (of vector spaces  though not of algebras) 
$H^*_G(X;\QQ)  \cong H^*(X;\QQ) \otimes H^*(BG;\QQ).$
Moreover, $\mu^{-1}(0)$ is $K$-invariant and its inclusion in $X^{ss}$ induces a homeomorphism 
\begin{equation}\label{fundamental}
\mu^{-1}(0)/K  \cong X/\!/G.
\end{equation} 
 When $X^{s}=X^{ss}$ the $G$-equivariant rational cohomology of $X^{ss}$ coincides with the ordinary rational cohomology of its geometric quotient $X^{ss}/G$, which is the GIT quotient $X/\!/G$, and we get expressions for the Betti numbers of $X/\!/G$ in terms of the equivariant Betti numbers of $X$ and the equivariant Betti numbers of the unstable GIT strata, which can be described inductively \cite{francesthesis}. In order to obtain the ring structure on the rational cohomology of $X/\!/G$, the surjectivity of the composition
 $$\kappa: H^*_G(X;\QQ) \to H^*_G(X^{ss};\QQ) \cong H^*(X/\!/G;\QQ)$$
 can be combined with Poincar\'e duality on $X/\!/G$ and the nonabelian localisation formulas for intersection pairings on $X/\!/G$ given in \cite{jeffreykirwan}.

Martin \cite{SM} used \eqref{fundamental} to obtain formulas for the  intersection pairings on the quotient $X/\!/G$ in a different way, by relating these pairings to intersection pairings on the associated quotient $X/\!/T_\CC$, where $T_\CC
$ is a maximal torus for $G$. He proved a formula expressing the rational cohomology ring of $X/\!/G$ in terms of the rational cohomology ring of $X/\!/T_\CC$ and an integration formula relating intersection pairings on the cohomology of $X/\!/G$ to corresponding pairings on $X/\!/T_\CC$. This integration formula, combined with methods from abelian localisation, leads to 
residue formulas for pairings on $X/\!/G$ which are closely related to those of \cite{jeffreykirwan} (see also \cite{vergne}). 
 
Note that, in the symplectic approach leading to \eqref{fundamental}, a maximal compact subgroup $K$ of $G$ and a $K$-invariant Fubini--Study K\"ahler metric $\omega$ are fixed and the corresponding moment map $\mu:X \to \mathfrak{k}^*$ for the $K$-action is used. 
In \cite{bkcoh} we introduced a new point of view as follows. Any other maximal compact subgroup of $G$ is given by $K' = g^{-1}Kg$ for some $g \in G$;
then $\omega_{K'} = g^*\omega$ is $K'$-invariant and if $\mu_K:X \to \mathfrak{k}^*$ exists then $\mu_{K'} = Ad^*(g^{-1}) \circ \mu_K \circ g$ is a $K'$-moment map with respect to $\omega_{K'}$, where $Ad^*$ denotes the co-adjoint action of $G$ on the complex dual $\mathfrak{g}^*$ of its Lie algebra and $\mathfrak{k}^*$ is embedded in $\mathfrak{g}^*$ as above. 
So to define a \lq moment map' for the $G$-action on $X$, instead of fixing a Fubini-Study K\"ahler form  $\omega$ it is natural to ask for a $G$-orbit $\Omega$ in   
$$\{ (K,\omega) \in \mathcal{K}_G \times \mbox{K\"ahler}(X) : \omega \mbox{ is $K$-invariant} \}$$
where $ \mbox{K\"ahler}(X)$ is the space of 
 K\"ahler forms on the complex manifold $X$ and 
 $$\mathcal{K}_G = \{K \, |\,  K \mbox{ is a maximal compact subgroup of } G\}.$$
 
In \cite{bkcoh} we call such a $G$-orbit $\Omega=\{(g^{-1}Kg,g^* \omega): g\in G\}$ a \emph{$G$-equivariant K\"ahler structure} on $X$. 
We define an \emph{$\Omega$-moment map} for the $G$-action on $X$    
 to be a smooth $G$-equivariant map 
$$\newmm: \Omega \times X \to \mathfrak{g}^*=\Hom_{\CC}(\mathfrak{k},\CC)$$
such that $\newmm(K,\omega, x) = \mu_{(K,\omega)}(x)$ for each $(K,\omega) \in \Omega$ 
 and $x \in X$, where
$\mu_{(K,\omega)}:X \to \mathfrak{g}^*$ is the composition of a moment map for the $K$-action on $(X,\omega)$ with the canonical embedding 
\begin{equation} \label{eqn:dual} \mathfrak{k}^* = \{\xi \in \mathfrak{g}^*: \xi(\mathfrak{k}) \subseteq \RR\}. \end{equation}
of the dual of the Lie algebra of $K$ in  $\mathfrak{g}^*$. Given a $G$-equivariant K\"ahler structure $\Omega$ on $X$ and a choice of maximal compact subgroup $K$ of $G$ with $(K,\omega) \in \Omega$, the existence and choice of an $\Omega$-moment map 
$\newmm: \Omega \times Y \to \mathfrak{g}^*$ for the $G$-action on $X$ is equivalent to the existence and choice of a moment map $\mu:X \to \lieks$ in the traditional sense for the $K$-action on the K\"ahler manifold $(X,\omega)$.
However,  this point of view with emphasis on the complex group action rather than the compact one, is a more natural one to take when extending the notion of a moment map to non-reductive linear algebraic groups.

\subsection{Non-reductive actions} In \cite{bkcoh} similar results are obtained for non-reductive actions. Let $X \subset \PP^n$ be a nonsingular complex projective variety with a linear action (with respect to an ample line bundle $L$) of a complex linear algebraic group $H=U \rtimes R$ with internally graded unipotent radical $U$. Let $\hU=U \rtimes \l(\CC^*) \subseteq H$ where $\l: \CC^* \to Z(R)$ is a grading 1-parameter subgroup; then $\l(S^1) \subseteq \l(\CC^*) \subseteq \hU$ is a maximal compact subgroup of $\hU$. Suppose that $H$ acts on $X$ via a homomorphism $\rho: H \rightarrow  G = \GL(n+1, \CC)$, 
 and the Levi subgroup $R$ is the complexification of a maximal compact subgroup $Q=K\cap H$ of $H$, where the unitary group $K=U(n+1)$ is a maximal compact subgroup of $G$. So $Q$ preserves the standard
$U(n+1)$-invariant hermitian inner product $\langle  \, ,\rangle$ on $\CC^{n+1}$.
We can use the corresponding Fubini-Study form $\omega$ on $\PP^n$ to define a $G$-equivariant K\"ahler structure 
\[\Omega_{G,\PP^n}=\{(g^{-1}Kg,g^*\omega):g\in GL(n+1)\}\] 
on $\PP^n$ which contains $(K,\omega)$. We can restrict the corresponding $\Omega$-moment map $\newmm: \Omega_{G,\PP^n} \times \PP^n \to \lieg^*$ to $\Omega_{H,X} \times X$ where 
\[\Omega_{H,X}=\{(K,\omega) \in \Omega_{G,\PP^n}:  K\cap H \in \mathcal{K}_H\}=\{(h^{-1}Kh,h^*\omega): h\in H\} \simeq H/Q\]
and compose it with the projection $p^*: \lieg^* \to \lieh^*$ to obtain an $\Omega$-moment map $\Omega_{H,X} \times X \to \lieh^*$ for the $H$-action on $X$ fitting into the following diagram:
\begin{equation}\label{uhatmomentmapH}
\xymatrix{\Omega_{H,X} \times X \ar[r]^-{\newmmxg} \ar[rd]_{\newmmxh}  &  \lieg^*=\liek^* \oplus i\liek^* \ar[d]^{p^*} \\
& \mathfrak{h}^*} 
\end{equation}
We obtain another $\Omega$-moment map by adding to $\newmmxh$ any $H$-equivariant map $\Omega \times X \to \mathfrak{h}^*$ which is constant and $K\cap H$-invariant on $\{(K,\omega)\}  \times X$ for every $(K,\omega)$ in the $H$-orbit $\Omega$. 
For $(K,\omega) \in \Omega_{H,X}$ the restriction $\mu^H_{(K,\omega)}: X \to \lieh^*$ to $\{(K,\omega)\} \times X$ is given, up to the addition of a $K \cap H$-invariant constant, by
\[\mu^H_{(K,\omega)}([x]) \cdot a =\frac{1}{2\pi i ||x||^2} \bar{x}^T \rho_*(a)x \in \CC \text{ for all } a \in H\]
where $\cdot$ denotes the natural pairing between $\lieh^*$ and $\lieh$.

Suppose now that the linearisation of the action of  $H$ on $X$ is well-adapted and $H$-stability=$H$-semistability in the strong sense of  Definition \ref{def:s=ss}).  Then it is shown in \cite{bkcoh} that for any $(K,\omega) \in \Omega_{X,H}$
\begin{enumerate}
\item $X^{s,H}_{\min+} = X^{ss,H}_{\min+} = \{ x \in X:0 \in \mathfrak{m}_{H,X,\Omega}(\Omega \times \{x\})\}=H(\mu_{(K,\omega)}^{H})^{-1}(0)$,
\item $0$ is a regular value of $\mu_{(K,\omega)}^{H}$ and
\item the embedding  of $(\mu_{(K,\omega)}^{H})^{-1}(0) $ in $X^{s,H}_{\min+}$ induces a diffeomorphism of orbifolds 
\begin{equation}\label{sympquotient}
(\mu_{(K,\omega)}^{H})^{-1}(0)/(K \cap H) \to X^{s,H}_{\min+}/H = X/\!/H. 
\end{equation}
\end{enumerate}

\begin{remark} In \cite{bkcoh} similar results are obtained  in the more general situation when $X$ is compact K\"ahler but not necessarily projective.
\end{remark}

In the present paper we will work with actions of the diffeomorphism group (see \S \ref{subsec:jetdiff})
\[\hU=\diff_k(1)=\left\{ \left(
\begin{array}{ccccc}
\alpha_1 & \alpha_2   & \alpha_3          & \ldots & \alpha_k \\
0        & \alpha_1^2 & 2\alpha_1\alpha_2 & \ldots & 2\alpha_1\alpha_{k-1}+\ldots \\
0        & 0          & \alpha_1^3        & \ldots & 3\alpha_1^2\alpha_ {k-2}+ \ldots \\
0        & 0          & 0                 & \ldots & \cdot \\
\cdot    & \cdot   & \cdot    & \ldots & \alpha_1^k
\end{array}
 \right)  \begin{array}{c} : \alpha_1,\ldots,\alpha_k \in \CC, \\ \alpha_1 \neq 0 \end{array} \right\}
\]
on projective varieties. Therefore we only state the remaining results of \cite{bkcoh} for the situation when $H=\hU=U \rtimes \l(\CC^*)$ 
acts via a homomorphism $\rho: \hU \to G = \GL(n+1)$ on a nonsingular projective variety $X \subset \PP^n$, such that if $K=U(n+1)$ then $K \cap \hU=\l(S^1)$ is the unique maximal compact subgroup of the one-parameter subgroup $\l(\CC^*)$. The Lie algebra of $\hU$ decomposes as a real vector space as 
\begin{equation}\label{decomp}
\hat{\lieu}=\RR \oplus i\RR \oplus \lieu
\end{equation}
where $\mathrm{Lie}(K \cap \hU)=i \RR$ and $\lieu$ is the Lie algebra of the complex unipotent group $U$. Diagram (\ref{uhatmomentmapH}) becomes
\begin{equation}\label{uhatmomentmap}
\xymatrix{\Omega_{\hU,X} \times X \ar[r]^-{\newmm} \ar[rd]^{\newmmxu} &  \lieg^*=\liek^* \oplus i\liek^* \ar[d] 
\\
& \hat{\lieu}^*=\RR \oplus i\RR \oplus \lieu^*} .
\end{equation}
We can allow the addition of a rational character (that is, a rational multiple of the derivative of the group homomorphism $\hU \to \CC^*$ with kernel $U$) to this \lq $\Omega$-moment map' $\newmmxu:X \to \hat{\lieu}^*$. Indeed, in the K\"ahler setting there is no reason not to allow real multiples, rather than only rational ones, here.

In this case, if  $\hU$-stability=$\hU$-semistability holds in the sense of Definition \ref{cond star} then \eqref{sympquotient} tells that for any $(K,\omega) \in \Omega_{\hU,X}$ the embedding $(\mu_{(K,\omega)}^{\hU})^{-1}(0) \hookrightarrow X^{s,\hU}_{\min+}$ induces a diffeomorphism of orbifolds 
\[(\mu_{(K,\omega)}^{\hU})^{-1}(0)/S^1 \overset{\simeq}{\to} X/\!/\hU.\] 
where $S^1=\l(S^1)=K \cap \hU$. We fix $(K,\omega)\in \Omega_{\hU,X}$, and drop it from the notation in the rest of this section, writing  $\mu_{\hU}$ for $\mu_{(K,\omega)}^{\hU}$ and $\mu_{\l(\CC^*)}$ fpr $\mu_{(K,\omega)}^{\l(\CC^*)}$ (which is a moment map in the usual sense for $\l(S^1)$ acting on $X$. 

Following Martin (see diagram (1.2) in  \cite{bkcoh}) we consider   
\begin{equation}\label{diagrammartin2}
\xymatrix{X/\!/\hU \overset{\simeq}{\leftarrow} \mu_{\hU}^{-1}(0)/S^1 \ar@{^{(}->}[r]^-{i} & \mu_{\l(\CC^*)}^{-1}(0)/S^1=X/\!/\CC^* .
}
\end{equation}
\begin{definition} 
 For a weight $\a$ of $\l(\CC^*) \subseteq \hU$, let $\CC_\a$  denote the corresponding $1$-dimensional complex representation of $\CC^*$ and let 
\[L_\a:=\mu_{\l(\CC^*)}^{-1}(0)\times_{S^1} \CC_\a \to X/\!/\CC^*,\]
denote the associated line bundle whose Euler class is denoted by $e(\a) \in H^2(X/\!/\CC^*)\simeq H_\CC^2(X)$.  
For a $\CC^*$-invariant complex subspace $\mathfrak{a} \subseteq \lieu$ let 
\[V_{\mathfrak{a}}= \mu_{\l(\CC^*)}^{-1}(0) \times_{S^1} \mathfrak{a} \to X/\!/\CC^*\]
denote the corresponding vector bundle. 
\end{definition}
\begin{proposition}[\cite{bkcoh}, Proposition 7.12 (i)]\label{propmartin}
The vector bundle $V_\lieu^* \to X/\!/\CC^*$ has a $C^{\infty}$-section $s$ which
is transverse to the zero section 
 and whose zero set is the submanifold
$\mu_{\hU}^{-1}(0)/S^1 \subseteq X/\!/\CC^*$. Therefore the $\CC^*$-equivariant normal bundle is
\[\mathcal{N}(i)\simeq V_\lieu^*. \]
\end{proposition}

This leads in \cite{bkcoh} to the following theorems:

\begin{theorem}[\cite{bkcoh}, Theorem 1.4]\label{thm:cohomologyrings} 
Let $X$ be a smooth projective variety endowed with a well-adapted action of $\hU=U \rtimes \CC^*$ such that $\hU$-stability=$\hU$-semistability holds. Then there is a natural ring isomorphism
\[H^*(X/\!/\hU,\QQ)\simeq \frac{H^*(X/\!/\CC^*,\QQ)}{ann(\mathrm{Euler}(V^*_\lieu))}.\]
Here $\mathrm{Euler}(V^*_\lieu) \in H^*(X/\!/\CC^*)$ is the Euler class of the bundle $V^*_\lieu$ and  
\[\mathrm{ann}(\mathrm{Euler}(V^*_\lieu))=\{c \in H^*(X/\!/\CC^*,\QQ)| c \cup \mathrm{Euler}(V^*_\lieu)=0\} \subseteq H^*(X/\!/\CC^*,\QQ).\]
is the annihilator ideal. 
\end{theorem}
\begin{theorem}[\cite{bkcoh}, Theorem 1.5]\label{thm:integration} 
Let $X$ be a smooth projective variety endowed with a well-adapted action of $\hU=U \rtimes \CC^*$ such that $\hU$-stability=$\hU$-semistability holds.
 Assume for simplicity that the stabiliser in $\hU$ of a generic $x \in X$ is trivial. Given a cohomology class $a \in H^*(X/\!/\hU)$  with a lift $\tilde{a}\in H^*(X/\!/\CC^*)$, then 
\[\int_{X/\!/\hU}a=\int_{X/\!/\CC^*}\tilde{a} \cup \mathrm{Euler}(V^*_\lieu),\]
where $\mathrm{Euler}(V^*_\lieu)$ is the cohomology class defined in Theorem \ref{thm:cohomologyrings}. Here we say that $\tilde{a}\in H^*(X/\!/\CC^*)$ is a lift of $a\in H^*(X/\!/\hU)$ if $a=i^*\tilde{a}$. 
\end{theorem}

\begin{remark} 
Theorem \ref{thm:integration}  
can be generalised to allow the triviality assumption for the stabiliser in $\hU$ of a generic $x \in X$ to be omitted; then the formula for $\int_{X/\!/\hU}a$ is multiplied by a strictly positive rational depending on the sizes of the stabilisers in $\hU$ and $\CC^*$ of a generic $x \in X$.
\end{remark}

Finally, we have residue formulas for the intersection pairings on the quotient $X/\!/\CC^*$. Note that $\hU$ is homotopy equivalent to $\CC^*$ and to $S^1$ so that $\hU$-equivariant cohomology is isomorphic to  $\CC^*$-equivariant cohomology  and to $S^1$-equivariant cohomology.
There are two surjective ring homomorphisms 
\[\kappa_{\CC^*}: H_{S^1}^*(X;\QQ) \to H^*(X/\!/\CC^*;\QQ) \text{  and  }  \kappa_{\hU}: H_{\hU}^*(X;\QQ)=H_{S^1}^*(X;\QQ) \to H^*(X/\!/\hU;\QQ)\]
from the $S^1$-equivariant cohomology of $X$ to the ordinary cohomology of the corresponding GIT quotients. The fixed points of the maximal compact subgroup $S^1$ of $\hU$ on $X\subseteq \PP^n$ correspond to the weights of the $\CC^*$ action on $X$, and since this action is well-adapted, these weights satisfy
\[\omega_{\min}=\omega_0<0<\omega_1 <\ldots <\omega_{n}.\]
We can represent elements of $H_{\hU}^*(X;\QQ)=H_{S^1}^*(X;\QQ)$ as polynomial functions on the Lie algebra of $\CC^*$ whose coefficients are differential forms on $X$ and which are equivariantly closed.

\begin{theorem}[\cite{bkcoh}, Theorem 7.15 and Corollary 7.16]
\label{jeffreykirwannonred}
Let $X$ be a nonsingular projective variety (assumed connected) endowed with a well-adapted action of $\hU=U \rtimes \CC^*$ such that $\hU$-stability=$\hU$-semistability holds (in the sense of Definitions \ref{cond star} and \ref{def:welladaptedaction}).   
Let $z$ be the standard coordinate on the Lie algebra of $\CC^*$. 
Given any $\hU$-equivariant cohomology class $\eta$ on $X$ represented by an equivariant differential form $\eta(z)$ whose degree is the dimension of $X/\!/\hU$, we have 
\[\int_{X/\!/\hU} \kappa_{\hU} (\eta) 
= 
n_{\hU}
\res_{z=0} \int_{F_{\min}}\frac{i_{F_{\min}}^* (\eta(z) \cup \mathrm{Euler}(V^*_\lieu)(z))}{\mathrm{Euler}(\mathcal{N}_{F_{\min}})(z)} dz \]
where $F_{\min} = Z_{\min}(X)$ is the union of those connected components of the fixed point locus $X^{\CC^*}$ on which the $S^1$-moment map takes its minimum value $\omega_{\min}$, and $n_{\hU}$ is a strictly positive rational number  which  depends only on $\hU$ and the size of the stabiliser in $\hU$ of a generic $x \in X$. 
 Here $\mathcal{N}_{F_{\min}}$ is the normal bundle to $F_{\min}$ in $X$ and $V^*_\lieu \to F_{\min}$ is the equivariant vector bundle given by $F_{\min} \times \lieu$.
\end{theorem}

\begin{remark} \label{rem:bkcoh}
In the situation of Theorem \ref{jeffreykirwannonred}, we can rewrite this formula as 
\begin{equation}  \label{niceversion} \int_{X/\!/\hU} \kappa_{\hU} (\eta) 
=   n_{\hU}
\res_{z=0} \int_{F_{\min}}\frac{i_{F_{\min}}^* (\eta(z)) 
}{\mathrm{Euler}(\mathcal{N}_{UF_{\min}})(z)} dz \end{equation}
where $\mathcal{N}_{UF_{\min}}$ is the normal to $UF_{\min}$ in $X$, since the restriction to $F_{\min}$ of the tangent bundle to $UF_{\min}$ is the direct sum of the tangent bundle to $F_{\min}$ and the trivial bundle $V_\lieu$ with fibre the Lie algebra of $U$. We can also identify $F_{\min}$ with $UF_{\min}/U$ and identify $\mathcal{N}_{UF_{\min}}$ with the normal to $UF_{\min}/U$ in $X^0_{\min}/U$.

Notice in addition that the surjection $ \kappa_{\hU}: H_{\hU}^*(X;\QQ)=H_{S^1}^*(X;\QQ) \to H^*(X/\!/\hU;\QQ)$ is the composition of two surjections
$$  H_{\hU}^*(X;\QQ)=H_{S^1}^*(X;\QQ) \to H_{\hU}^*(X^0_{\min};\QQ)=H_{S^1}^*(X^0_{\min};\QQ) \to  H^*(X/\!/\hU;\QQ),$$
and that $F_{\min} \subseteq X^0_{\min}$. Thus to calculate any intersection pairing 
$$\langle \kappa_{\hU}(\alpha),\kappa_{\hU}(\beta) \rangle = \int_{X/\!/ \hU} \kappa_{\hU}(\alpha \beta)$$ using (\ref{niceversion}), we only need to know the restriction to $X^0_{\min}$ of $S^1$-equivariant cohomology classes $\alpha$ and $\beta$ which represent $\kappa_{\hU}(\alpha)$ and $\kappa_{\hU}(\beta)$.
\end{remark}

\section{Non-reductive GIT compactification of the jet differentials bundle}\label{quotient}\label{sec:snowman}

Let $X$ be a nonsingular complex projective variety of dimension $n$. In this section we will use non-reductive GIT to construct a projective completion of the quasi-projective quotient $J_k^\reg X/\diff_k(1)$, introduced in \S \ref{subsec:jetdiff}. Recall that $\diff_X$ denotes the principal $\diff_k(n)$-bundle over $X$ formed by all local polynomial coordinate systems on $X$, 
 and that
\[\calx_k^{\text{reg}}=J_k^{\text{reg}}X/\diff_k(1) \cong \diff_X \times_{\diff_k(n)} X_k^{\text{reg}}\]
where $X_k^{\text{reg}}=J_k^{\text{reg}}(1,n)/\diff_k(1))$ is isomorphic to each fibre of $\calx_k^{\text{reg}}$ over $X$. 
We will construct a projective completion  
\[\calx_k^{\GIT}
 \cong
 \diff_X \times_{\diff_k(n)} X_k^{\GIT}\]
of $\calx_k^{\text{reg}}$ where the fibre $X_k^{\GIT}$ 
 of $\calx_k^{\GIT}$ over $X$
is a non-reductive GIT quotient of a blow-up of the projective space
\[ \PP(\CC \oplus J_k(1,n)) = \PP(\CC \oplus \Hom(\CC^k, \CC^n))=\PP[x:v_1:v_2:\ldots :v_k]\]
by the action of 
\[\hU=\diff_k(1) = \left\{ \left(
\begin{array}{ccccc}
\alpha_1 & \alpha_2   & \alpha_3          & \ldots & \alpha_k \\
0        & \alpha_1^2 & 2\alpha_1\alpha_2 & \ldots & 2\alpha_1\alpha_{k-1}+\ldots \\
0        & 0          & \alpha_1^3        & \ldots & 3\alpha_1^2\alpha_ {k-2}+ \ldots \\
0        & 0          & 0                 & \ldots & \cdot \\
\cdot    & \cdot   & \cdot    & \ldots & \alpha_1^k
\end{array} 
 \right)  \begin{array}{c}  : \alpha_1,\ldots,\alpha_k \in \CC, \\ \alpha_1 \neq 0 \end{array} \right\}.
\]
Here $v_1,\ldots, v_k \in \CC^n$ are vectors representing the columns of a matrix $M \in \Hom(\CC^k, \CC^n)$, and $x$ is the compactifying coordinate, while $\hU$ acts via the right action
\[ [x:M]\cdot \hat{u}=[x:M\hat{u}] \text{ for } \hat{u} \in \hU\]
or equivalently via the left action
\[\hat{u}\cdot [x:M]=[x:M(\hat{u})^{-1}] \text{ for } \hat{u} \in \hU.\]

\begin{remark} \label{rem:projcompletion}
We need to be careful here with the construction of $\calx_k^{\GIT}$ as $ \diff_X \times_{\diff_k(n)} X_k^{\GIT}$. This is because $\diff_k(n)$ does not act naturally on the projective space
$\PP(\CC \oplus J_k(1,n)) = \PP(\CC \oplus \Hom(\CC^k, \CC^n))$ or its blow-up, since its action on $J_k(1,n)$ is not linear, and the identification of a fibre of $J_kX$ with $\Hom(\CC^k, \CC^n))$ depends on a choice of local coordinates on $X$. Nonetheless we will find that $\diff_k(n)$ does act on $X_k^{\GIT}$ compatibly with its action on $J_k^\reg (1,n)/\diff_k(1)$.
\end{remark}

We want to apply the results of non-reductive GIT described in $\S$3, which are stated for left actions. For this we need a one-parameter subgroup $\lambda:\CC^* \to \hU$ whose adjoint action on the Lie algebra of $U$ has only strictly positive weights; we can take
\[\lambda(t)=\left(
\begin{array}{ccccc}
t^{-1} &0   & 0          & \ldots & 0 \\
0        & t^{-2} & 0 & \ldots & 0 \\
0        & 0          & t^{-3}        & \ldots & 0 \\
0        & 0          & 0                 & \ldots & \cdot \\
\cdot    & \cdot   & \cdot    & \ldots &t^{-k}
\end{array}
 \right),
\]
and then 
\[\lambda(t)  \cdot [x:M]=[x:M  \left(
\begin{array}{ccccc}
t &0   & 0          & \ldots & 0 \\
0        & t^{2} & 0 & \ldots & 0 \\
0        & 0          & t^{3}        & \ldots & 0 \\
0        & 0          & 0                 & \ldots & \cdot \\
\cdot    & \cdot   & \cdot    & \ldots &t^{k}
\end{array}
 \right) ].
\]
 The weights of this (left)  action of the one-parameter subgroup of $\hU$ defined by $\lambda$  are $\{0, 1, 2, \ldots, k\}$. The minimal weight space is the point
\[Z_{\min}=\{ [1:0:\ldots :0] \} \]
and the $U$-stabiliser of this point is $
U$, so this action does not satisfy the \lq semistability coincides with stability' condition (see Definition \ref{cond star}) which is required for the main results described in $\S\S$3 and 4.
We wish to construct a blow-up of 
the projective space
$\PP(\CC \oplus J_k(1,n)) = \PP(\CC \oplus \Hom(\CC^k, \CC^n))$
in order to satisfy this condition.  Following \cite{bdhk1} we try blowing up the projective space $\PP(\CC \oplus \Hom(\CC^k, \CC^n))$ along $Z_{\min}$ to  get 
\[\tPP=\blow_{[1:0:\ldots :0]}\PP(\CC \oplus \Hom(\CC^k, \CC^n))=\{([x:v_1,\ldots, v_k],[w_1,\ldots w_k]):w_i\otimes v_j=w_j\otimes v_i \mbox{ for }  1\le i<j\le k\}\]
embedded in $\PP^{kn} \times \PP^{kn-1} \subseteq \PP^{(kn+1)kn-1}$. We fix an ample linearisation $L=\calo_{\PP^{kn}}(1) \otimes \calo_{\PP^{kn-1}}(1)$ on $\PP^{kn} \times \PP^{kn-1}$ and restrict it to $\tPP$. The minimal weight space for the action of $\lambda(\CC^*)$ on $\tPP$ is the intersection $\tZmin$ of the exceptional divisor $E$ and the strict transform of $\PP[x:v_1:0:\cdots:0] \subseteq \PP(\CC \oplus \Hom(\CC^k, \CC^n))$:
$$ \tZmin=\{([1:0:\ldots :0],[w_1:0:\ldots:0]): w_1 \in \CC^n, \,\, w_1 \neq 0 \}\subset E \subset \tPP. $$
The $U$-stabiliser of any point in $\tZmin$ is trivial, and hence stability coincides with semistability for the induced $\hU$ action on $\tPP$. Thus we can apply Theorem \ref{mainthm} to obtain a  non-reductive GIT quotient 
\begin{equation} \label{XGIT}  X_k^{\GIT}=\tPP/\!/\hU = \tPP^{s,\hU}/\hU \end{equation}
with respect to a well-adapted shift of $L$. Then $\tPP/\!/\hU$ is a projective variety and is a geometric quotient by $\hU$ of the open subvariety $\tPP^{s,\hU} = \tPP^{0}_{\min} \setminus U Z_{\min}(\tPP)$ of $\tPP$, which contains $J_k^{\text{reg}}(1,n)$. Thus $\tPP/\!/\hU$ is a projective completion of $J_k^{\text{reg}}(1,n)/\diff_k(1)$.

\begin{remark} \label{rem:projcompletion2}
The open subvariety $\tPP^{s,\hU} = \tPP^{0}_{\min} \setminus U Z_{\min}(\tPP)$ of $\tPP$ contains $J_k^{\text{reg}}(1,n)$ and is contained in the blow-up of the affine space $J_k(1,n)\cong \Hom(\CC^k, \CC^n)$ at the origin; here the origin represents the trivial jet which is identically 0. The action of $\diff_k(n)$ on $J_k(1,n)$ is not linear but it fixes the origin and hence there is an induced action on the blow-up of $J_k(1,n)$ at the origin. This action preserves $\tPP^{s,\hU} = \tPP^{0}_{\min} \setminus U Z_{\min}(\tPP)$ and commutes with the action of $\diff_k(1)$, so there is an induced action of $\diff_k(n)$ on the geometric quotient $X_k^{\GIT}=\tPP/\!/\hU = \tPP^{s,\hU}/\hU$.
\end{remark}

\begin{definition} \label{defn:XGIT}
Let $\calx_k^{\GIT} =   \diff_X \times_{\diff_k(n)} X_k^{\GIT} $
where $X_k^{\GIT}$ is defined at (\ref{XGIT}) above.
\end{definition}

Thus $\calx_k^{\GIT}$ fibres over $X$ with fibre $X_k^{\GIT} $, and can be regarded as a projective completion of $\calx_k^{\text{reg}}=J_k^{\text{reg}}X/\diff_k(1) \cong \diff_X \times_{\diff_k(n)} X_k^{\text{reg}}$. It can also be identified with the geometric quotient by $\diff_k(1)$ of an open subvariety of the blow-up of the jet bundle $J_kX$ along its zero section. The jet bundle $J_kX$ is not a vector bundle over $X$, since its structure group is $\diff_k(n)$ which is not a subgroup of $\GL(n)$, but nevertheless it has a well-defined zero section.

\begin{remark} \label{lineNX}
Recall from Remark \ref{lineN} that we have fixed a positive integer $N$ such that there is an ample line bundle $\calo_{X_k^{\GIT}}(1)$ on $X_k^{\GIT}$ which pulls back to $L^{\otimes N}$ on $\tPP^{s,\hU}$. Here $\calo_{X_k^{\GIT}}(1)$ can be constructed as the geometric quotient of the restriction of $L^{\otimes N}$ to $\tPP^{s,\hU}$ by the well-adapted linear action of $\diff_k(1)$, where $N$ is sufficiently divisible that for every $x \in \tPP^{s,\hU}$ the finite stabiliser $\stab_{\hU}(x)$ acts trivially on the fibre of $L^{\otimes N}$ at $x$. 
 Then sections of powers of $L^{\otimes N}$ which are $\hU$-invariant for the well-adapted linearisation define sections of powers of $\calo_{X_k^{\GIT}}(1)$. More precisely, if $j $ is a natural number, sections of $\calo_{X_k^{\GIT}}(j)$ pull back to $\hU$-invariant sections of the restriction of  $L^{\otimes Nj}$ to $\tPP^{s,\hU}$, and these all extend to $\hU$-invariant sections of  $L^{\otimes Nj}$
  over $\tPP$; here  every $\hU$-invariant section of 
 $L^{\otimes Nj}$
  over $\tPP$ vanishes on the complement of $\tPP^{s,\hU}$. Since $\tPP^{s,\hU}$ is a $\diff_k(n) \times \diff_k(1)$-invariant open subset of the blow-up of $J_k(1,n)$ at 0 (where $\hU = \diff_k(1)$), it follows that sections of the restriction to the blow-up of $J_k(1,n)$ at 0 of $L^{\otimes Nj}$, which are $\hU$-invariant for the well-adapted linearisation, define sections of  $\calo_{X_k^{\GIT}}(j)$. Since $L$ is given by $\calo_{\PP^{kn}}(1) \otimes \calo_{\PP^{kn-1}}(1)$ where the restriction of the line bundle $\calo_{\PP^{kn}}(1)$ to $J_k(1,n)$ is trivial, the restriction of $L$ to the blow-up of $J_k(1,n)$ at 0 is $\calo(-E)$ where $E$ is the exceptional divisor. Furthermore sections of $\calo_{X_k^{\GIT}}(j)$ pull back to $\hU$-invariant sections of the restriction of  $L^{\otimes Nj}$ to $\tPP^{s,\hU}$ which extend by zero to $\hU$-invariant sections of  $\calo(-NjE)$
  over  the blow-up of $J_k(1,n)$ at 0. This gives us an identification of $$\bigoplus_{j=1}^m H^0(X_k^{\GIT},\calo_{X_k^{\GIT}}(j)) = H^0(X_k^{\GIT}, \bigoplus_{j=1}^m \calo_{X_k^{\GIT}}(j))$$ with a subspace of the space of  $\hU$-invariant sections of  $\bigoplus_{j=1}^m \calo(-NjE)$
  over  the blow-up of $J_k(1,n)$ at 0 which are polynomial of degree at most $Nm$ in the coordinates on $J_k(1,n)$. 

Recall also that since $\tPP^{s,\hU}$ is a $\diff_k(n) \times \diff_k(1)$-invariant open subset of the blow-up of $J_k(1,n)$ at 0,  there is an identification
$$ \tilde{J}_k^{s,\hU}X \cong \diff_X \times_{\diff_k(n)} \tPP^{s,\hU}$$
where $ \tilde{J}_k^{s,\hU}X$ is a $\diff_k(1)$-invariant open subset of the blow-up $\tilde{J}_kX$ of $J_kX$ along its zero section. 
There is thus a relatively ample line bundle $\calo_{\calx_k^{\GIT}}(1)$ on $\calx_k^{\GIT} \to X$, constructed as the geometric quotient of 
$$\diff_X \times_{\diff_k(n)} L^{\otimes N}|_{\tPP^{s,\hU}}$$
by the action of $\diff_k(1)$, which   
pulls back on each fibre of $\calx_k^{\GIT} \to X$
 to  $L^{\otimes N}$ (or equivalently to $\calo(-NE)$)  on $\tPP^{s,\hU} = \tPP^{0}_{\min} \setminus U Z_{\min}(\tPP)$.   
 Thus sections of $\bigoplus_{j=1}^m \calo_{\calx_k^{\GIT}}(j)$ pull back on  each fibre of $\calx_k^{\GIT} \to X$
 to restrictions of $\hU$-invariant sections of  $\bigoplus_{j=1}^m \calo(-NjE)$
  over  the blow-up of $J_k X$ along its zero section, which are polynomial of degree at most $Nm$ in the coordinates on the corresponding fibre of $J_kX \to X$.
 
 The first Chern class $c_1(\calo_{\calx_k^{\GIT}}(1))$ restricts to $c_1(\calo_{X_k^{\GIT}}(1))$ on the fibre $X_k^{\GIT}$, and this is given in the notation of Theorem  \ref{jeffreykirwannonred} and Remark \ref{rem:bkcoh} by 
 \begin{equation} \label{firstChern}  c_1(\calo_{X_k^{\GIT}}(1)) = \kappa_{\hU}|_{\tPP^0_{\min}} (c_1^{S^1}(L^{\otimes N}|_{\tPP^0_{\min}})) \end{equation}
 where $c_1^{S^1}(L^{\otimes N})$ is the $S^1$-equivariant first Chern class of $L^{\otimes N}$, and $c_1^{S^1}(L^{\otimes N}|_{\tPP^0_{\min}}) = c_1^{S^1}(L^{\otimes N})|_{\tPP^0_{\min}}$ is its restriction to $\tPP^0_{\min}$. From Remark \ref{rem:projcompletion2}, since $L=\calo_{\PP^{kn}}(1) \otimes \calo_{\PP^{kn-1}}(1)$ where the restriction of the line bundle $\calo_{\PP^{kn}}(1)$ to $J_k(1,n)$ is trivial, $c_1^{S^1}(L^{\otimes N})$ has the same restriction to the blow-up of $J_k(1,n)$ as the power 
 $$c_1^{S^1}( \calo_{\PP^{kn-1}}(1))^N$$ of the $S^1$-equivariant first Chern class of the exceptional divisor $E$ for the  blow-up $\tilde{J}_kX$ of $J_kX$ along its zero section.
\end{remark}

Now let's consider $X_k^{\GIT}$  in the simple cases when $k=2,3$.

\subsection{Description of $X_2^{\GIT}$}\label{subsection:example} Recall from \S \ref{quotient} that the reparametrisation group $\hU=\diff_2(1) =\left\{u(\a_1,\a_2)=\left(\begin{array}{cc} \a_1 & \a_2 \\ 0 & \a_1^2 \end{array}\right): \a_1 \in \CC^*,\a_2 \in \CC \right\}$ acts on the projective space
$\PP(\CC \oplus \Hom(\CC^2, \CC^n))=\PP[x:f':f'']$ via 
\[[x:f':f''] \cdot u(\a_1,\a_2)=[x:\a_1f':\a_1^2f''+\a_2f']\]
Note that $f',f'' \in \CC^n$ form the columns of an $n \times 2$ matrix $M \in \Hom(\CC^2, \CC^n)$, and $\diff_2(1)$ acts on the right via matrix multiplication.
Thus  the grading $1$-parameter subgroup is $\l(t)=\left(\begin{array}{cc} t & 0 \\ 0 & t^2 \end{array} \right)$, whose weight on the $1$-dimensional Lie algebra $\mathrm{Lie}(U)$ is $1$. The minimal weight space  $Z_{\min} = \{[1:0:0]\}$ consists of a single point and the blow-up of $\PP(\CC \oplus \Hom(\CC^k, \CC^n))$ at this point is $\tPP=\blow_{[1:0:0]}\PP=\{([x:v_1:v_2],[w_1, w_2]):w_1\otimes v_2=w_2\otimes v_1\}$. The minimal weight space in $\tPP$ is 
\[\tZmin=\{([1:0:0],[w_1:0]): w_1 \in \CC^n, \,\, w_1 \neq 0 \}\subseteq E \subseteq \tPP,\] 
which sits in the exceptional divisor $E$ for the blow-up, and $\tZmin \simeq \PP^{n-1}$ is a projective space. The maximal unipotent is $$U=\left\{ \left(\begin{array}{cc} 1 & \a_2 \\ 0 & 1 \end{array}\right) : \a_2 \in \CC \right\},$$ and the $U$-stabiliser of any point in $\tZmin$ is trivial. Hence stability coincides with semistability for the induced $\diff_2(1)$ action on $\tPP$. The projection $p: \tPP_{\min}^0 \to \tZmin$ 
 has affine fibres and $U$ acts fibrewise, so $p$ induces a fibration $\bar{p}: X_2^{\GIT} = \tPP /\!/ \diff_2(1) \to \tZmin$ where the fibre over $z$ is isomorphic to a (weighted) projective space $((p^{-1}(z)\setminus Uz)/U)/\l(\CC^*)$. More concretely, 
\[p^{-1}([1:0:0],[w_1:0]) \simeq \{[\lambda w_1:w_2] \in \Hom(\CC^2, \CC^m): \l \in \CC, w_2\in \CC^n\} \simeq \CC^{n+1}\]
Assume $w_1=(w_1^1,\ldots, w_1^n)^T$ with $w_1^1 \neq 0$. Then $(p^{-1}(z)/U)$ can be identified with a slice for the $U$ action on $p^{-1}(z)$: for any $(\l w_1. w_2) \in p^{-1}(z)$
there is a unique $\left(\begin{array}{cc} 1 & \a_2 \\ 0 & 1 \end{array}\right) \in U$ such that 
\begin{equation}\label{slice}
\left(\begin{array}{cc} \l w_1^1 & w_2^1 \\ \l w_1^2 & w_2^2 \\ . &. \\ \l w_1^n & w_2^n \end{array} \right) \cdot \left(\begin{array}{cc} 1 & \a_2 \\ 0 & 1 \end{array}\right)=\left(\begin{array}{cc} \l w_1^1 & 0 \\ \l w_2^2 & \tilde{w}_2^2 \\ . &. \\ \l w_1^2 & \tilde{w}_2^n \end{array} \right)
\end{equation}
Hence $p^{-1}(z)/U \simeq \CC^n$ with coordinates $\l, \tilde{w}_2^2, \ldots, \tilde{w}_2^n$. The weight of the well-adapted shift of $\l(\CC^*)$ on $\l$ is $-\epsilon<0$ and $1-\epsilon>0$ on $\tilde{w}_2^2, \ldots, \tilde{w}_2^n$, hence $(p^{-1}(z)/U)/\!/\CC^*=\PP^{n-1}$. Thus $X_k^{\GIT}$ fibres over $\tZmin =\PP^{n-1}$ with fibres isomorphic to $\PP^{n-1}$. 

Note that in this case $N=1$, and sections of $\calo_{X_2^{\GIT}}(1)$ are given by coordinates on this affine $U$-slice. In \eqref{slice} above 
\[\tilde{w}_2^i=w_2^i-\frac{w_2^1}{w_1^1}w_1^i\]
and hence the minor $\left|\begin{array}{cc} w_1^1 & w_2^1 \\ w_1^i & w_2^i \end{array} \right|=w_2^i w_1^1-w_1^i w_2^1$ gives such a section. This is indeed invariant under $\diff_k(1)$, its weighted degree is $1+2=3$, corresponding to the bundle $E_{2,3}$ of semi-invariant jet differentials. The sections of $\calo_{X_2^{\GIT}}(m)$ are given by degree $m$ polynomials in the $2 \times 2$ minors of $(w_1,w_2)$.

\subsection{Description of $X_3^{\GIT}$} When $k=3$, the reparametrisation group 
 $$\hU=\diff_3(1)=\left\{u(\a_1,\a_2,\a_3)=\left(\begin{array}{ccc} \a_1 & \a_2 & \a_3 \\ 0 & \a_1^2 & 2\a_1\a_2 \\ 0 & 0 & \a_1^3\end{array}\right): \a_1 \in \CC^*,\a_2,\a_3 \in \CC \right\}$$ acts on the projective space
$\PP(\CC \oplus \Hom(\CC^3, \CC^n))=\PP[x:v_1:v_2:v_3]$ via 
\[[x:v_1:v_2:v_3] \cdot u(\a_1,\a_2,\a_3)=[x:\a_1v_1:\a_1^2v_2+\a_2v_1:\a_1^3c_3+2\a_1\a_2v_2+\a_3v_1].\]
With $\l(t)=\left(\begin{array}{ccc} t & 0 & 0 \\ 0 & t^2 & 0\\ 0 & 0 & t^3\end{array} \right)$ the minimal weight space on the blow-up $\tPP$ is 
\[\tZmin=\{([1:0:0:0],[w_1:0:0]): w_1 \in \CC^n, \,\, w_1 \neq 0 \}\subseteq E \subseteq \tPP.\]
The fibre of the projection $p$ is 
\[p^{-1}([1:0:0:0],[w_1:0:0]) \simeq \{[\lambda w_1:w_2:w_3] \in \Hom(\CC^3, \CC^m): \l \in \CC, w_2,w_3\in \CC^n\} \simeq \CC^{2n+1}\]
Assume $w_1^1 \neq 0$, then there is a unique $\left(\begin{array}{ccc} 1 & \a_2 & \a_3\\ 0 & 1 & 2\a_2 \\ 0 & 0 & 1\end{array}\right) \in U$ which eliminates the first row: 
\begin{equation}\label{eliminate3}
\left(\begin{array}{ccc} \l & w_2^1 & w_3^1 \\ w_1^2 & w_2^2 & w_3^2 \\ . &. & .\\ w_1^n & w_2^n & w_3^n \end{array} \right) \cdot \left(\begin{array}{ccc} 1 & \a_2 & \a_3 \\ 0 & 1 & 2\a_2 \\ 0 & 0 & 1 \end{array}\right)=\left(\begin{array}{ccc} \l & 0 & 0 \\ w_1^2 & \tilde{w}_2^2 & \tilde{w}_3^2 \\ . &. & . \\ w_1^n &  \tilde{w}_2^n & \tilde{w}_3^n \end{array} \right)
\end{equation}
Hence $p^{-1}(z)/U \simeq \CC^{2n-1}$ with coordinates $\l, \tilde{w}_2^2, \ldots, \tilde{w}_2^n, \tilde{w}_2^2, \ldots, \tilde{w}_2^n$ and adapted $\l(\CC^*)$-weight vector $(-\epsilon, 1-\epsilon,\ldots, 1-\epsilon, 2-\epsilon, \ldots, 2-\epsilon)$. Therefore $(p^{-1}(z)/U)/\!/\CC^*=\PP[1,\ldots, 1,2,\ldots, 2]$ is a weighted projective space. Thus $X_3^{\GIT}$ fibres over $\tZmin =\PP^{n-1}$ with fibres isomorphic to weighted projective spaces.  

The sections of $\calo_{X_3^{\GIT}}(1)$ come again from the coordinates on the constructed affine $U$-slice: $\tilde{w}_2^i$ correspond to $2 \times 2$ minors, and the new invariants come from $\tilde{w}_3^i$. From \eqref{eliminate3} one obtains 
\[\a_2=-\frac{w_2^1}{w_1^1}, \ \ \ \ \a_3=\frac{2u_1^2-v_1w_1}{w_1^2}\]
and hence 
\[\tilde{w}_3^i=w_3^i-\frac{2w_2^iw_2^1+w_1^iw_3^1}{w_1^1}+\frac{2w_1^i(w_2^1)^2}{(w_1^1)^2}\]
which gives the invariant section
\begin{equation}\label{newinvs}
w_1^1 \left|\begin{array}{cc} w_1^1 & w_3^1 \\ w_1^i & w_3^i \end{array} \right| +2w_2^1 \left|\begin{array}{cc} w_1^1 & w_2^1 \\ w_1^i & w_2^i \end{array} \right|
\end{equation}
To get all sections of powers of $\calo_{X_3^{\GIT}}(1)$, we run the upper index pair $(1,i)$ in this formula over all pairs $(i,j)$, $1\le i,j \le 3$, $i\neq j$. One can check that these are indeed invariant under $U$ with the weighted degree  $1+1+3=2+2+1=5$, corresponding to the bundle $E_{3,5}$ of semi-invariant jet differentials.

\subsection{$X_k^{\GIT}$ for higher $k$}    
For general $k\geq 2$ a similar argument gives the following picture:

\begin{proposition}\label{prop:xkgit} $X_k^{\GIT}$ is a fibration over $\tZmin \simeq \PP^{n-1}$ with fibres isomorphic to a weighted projective space $\PP[1,\ldots, 1,2,\ldots, 2, \ldots, k-1,\ldots, k-1]$ with all weights repeated $n-1$ times. 
Sections of $\calo_{X_k^{\GIT}}$ are given by homogenisation of coordinates on the canonical section constructed as above for $k= 2,3$.
\end{proposition}

Note that
\begin{equation} \label{dimxgit} \dim X_k^{\GIT} = k(n-1) \mbox{ and } \dim \calx_k^{\GIT} = k(n-1) +n. \end{equation}

\section{Reducing hyperbolicity to intersection theory}\label{sec:existence}

Now let $X\subseteq \PP^{n+1}$ be a smooth projective hypersurface of dimension $n$ and degree $\deg(X)=d$ in $ \PP^{n+1}$. Recall from Remark \ref{lineNX} that there is an integer $N>0$ and an ample line bundle $\calo_{X_k^{\GIT}}(1)$ on $X_k^{\GIT}$ which pulls back to $L^{\otimes N}$ on $\tPP^{s,\hU}$, constructed as the geometric quotient of the restriction of $L^{\otimes N}$ to $\tPP^{s,\hU}$ by the well-adapted linear action of $\diff_k(1)$. Here $N$ is sufficiently divisible that for every $x \in \tPP^{s,\hU}$ the finite stabiliser $\stab_{\hU}(x)$ acts trivially on the fibre of $L^{\otimes N}$ at $x$. 
The fibre $X_k^{\GIT}$ of $\calx_k^{\GIT} \to X$ is a non-reductive GIT quotient   $\tPP/\!/\hU$ where  $\hU = \Diff_k$, and is a projective variety of dimension $\dim(X_k^{\GIT})=k(n-1)$ with at worst orbifold singularities.
Moreover
 there is an identification
$$ \tilde{J}_k^{s,\hU}X \cong \diff_X \times_{\diff_k(n)} \tPP^{s,\hU}$$
where $ \tilde{J}_k^{s,\hU}X$ is a $\diff_k(1)$-invariant open subset of the blow-up $\tilde{J}_kX$ of $J_kX$ along its zero section, 
and a relatively ample line bundle $\calo_{\calx_k^{\GIT}}(1)$ on $\calx_k^{\GIT} \to X$, constructed as the geometric quotient of 
$$\diff_X \times_{\diff_k(n)} L^{\otimes N}|_{\tPP^{s,\hU}}$$
by the action of $\diff_k(1)$, which   
pulls back on each fibre of $\calx_k^{\GIT} \to X$
 to  $L^{\otimes N}$ (or equivalently to $\calo(-NE)$ where $E$ is the exceptional divisor for the blow-up)  on $\tPP^{s,\hU} = \tPP^{0}_{\min} \setminus U Z_{\min}(\tPP)$.   

\begin{remark}
Our proof of Theorem \ref{mainthmtwo} will involve checking that an integral over $\calx_k^{\GIT}$ of a cohomology class depending on $N$ is strictly positive. However it will turn out that this integral is independent of $N$ up to multiplication by factors of $N$, so its positivity does not depend on $N$.
\end{remark}

First we need to relate the direct image sheaf $\pi_*\calo_{\calx_{k}^{\GIT}}(m)$ to the Demailly--Semple  bundles $E_{k,j}$ of invariant jet differentials of order $k$ and weighted degree $j$.

\begin{proposition}\label{directimage2} 
Let $\pi: \calx_{k}^{\GIT} \to X$ denote the projection and $N$ be as defined as in Remarks \ref{lineN} and \ref{lineNX}. If $k>1$ then the direct image sheaf 
\begin{equation*}
\pi_*\calo_{\calx_{k}^{\GIT}}(m) \subseteq  \calo(E_{k,\le N^2km})
\end{equation*}
is a subsheaf of the sheaf of holomorphic sections of $E_{k,\le N^2km}=\oplus_{j=0}^{N^2km}E_{k,j}$.
\end{proposition}

\proof
The result to be proved is local on $X$ so it suffices to work in local coordinates in a neighbourhood of $x \in X$. Then the derivatives $f'(0),\ldots, f^{(k)}(0)$ of holomorphic maps $f: (\CC,0) \to (X,x)$ with respect to these local coordinates on $X$ provide coordinates on the fibres of the affine bundle $J_kX \to X$.

Recall from $\S$2.3 that $E_{k,j}$ is a subbundle of the  bundle 
$$ 
\Diff_X \times_{\diff_k(n)} \cale_{k,j} 
$$ 
over $X$, where $ \cale_{k,j}$ is the space of weighted degree $j$ complex-valued polynomials $Q(f'(0),\ldots, f^{(k)}(0))$ on the space $J_k(1,n)$ of $k$-jets of holomorphic maps $f: (\CC,0) \to (\CC^n,0)$. Here $ \cale_{k,j}$ has commuting left and right actions of $\diff_k(n)$ and $\diff_k(1)$, and 
$$ 
E_{k,j}=
\Diff_X \times_{\diff_k(n)} \cale_{k,j}^{U_k} 
$$
where $U = U_k$ is the unipotent radical of $\diff_k(1)$ and $\hU = \diff_k(1)$.

Recall also from Definition \ref{defn:XGIT} that
$$  \calx_k^{\GIT} =   \diff_X \times_{\diff_k(n)} X_k^{\GIT} $$
where $X_k^{\GIT}  = \tPP^{s,\hU}/\hU $ is a non-reductive GIT quotient 
 which is a projective completion of $J_k^{\text{reg}}(1,n)/\diff_k(1)$. Here $\tPP^{s,\hU}$ 
 is an open subvariety of the blow-up of the affine space $J_k(1,n)\cong \Hom(\CC^k, \CC^n)$ at the origin. The action of $\diff_k(n)$ on $J_k(1,n)$ is not linear but it fixes the origin and hence acts on this blow-up. 
 It preserves $\tPP^{s,\hU}$ and commutes with the action of $\diff_k(1)$, so there is an induced action of $\diff_k(n)$ on the geometric quotient $X_k^{\GIT}= 
 \tPP^{s,\hU}/\hU$.

By Remark \ref{lineNX}, a section $\s$ of $ \bigoplus_{j=1}^m
 \calo_{\calx_k^{\GIT}}(j)$ pulls back over  each fibre of $\calx_k^{\GIT} \to X$
 to the restriction of a $\diff_k(1)$-invariant section (with respect to the well-adapted linearisation) $\tilde{\s}$ of  $\bigoplus_{j=1}^m \calo(-NjE)$
  over  the blow-up of $J_k X$ along its zero section, where $E$ is the exceptional divisor for this blow-up and $\tilde{\s}$ is polynomial of degree at most $Nm$ in the coordinates $f'(0),\ldots, f^{(k)}(0)$ on the corresponding fibre of $J_kX \to X$, 
when we use the local coordinates on $X$ to identify the $j$th column of $J_k(1,n)\cong \Hom(\CC^k, \CC^n)$ with the $j$th derivative $f^{(j)}(0)$ of a holomorphic jet $f$ at $x$.  Then $\tilde{\s}$ has the form
\[\sum_{i=0}^{Nm} Q_i(f',f'',\ldots, f^{(k)})\]
where 
 $Q_i$ is a homogeneous polynomial of degree $i$ in its entries $f'(0),\ldots, f^{(k)}(0)$ and is invariant under the well-adapted action of $\diff_k(1)$. 
 For this well-adapted action the weighted degree of $f^{(j)}(0)$ is $jN-\chi$ for some $\chi>0$, determined by the well-adapted shift of $L^{\otimes N}$. Well-adaptedness means that  
\[\omega_{\min}=0+N-\chi < 0 \ll \omega_{\min+1}=0+2N-\chi\]
and hence $N<\chi \ll 2N$ is slightly bigger than $N$. A well-adapted weighted-homogeneous degree $d$ monomial  $(f')^{i_1} \ldots (f^{(k)})^{i_k}$ of degree $i_1 + \ldots + i_k$ satisfies  
\[i_1(N-\chi)+\ldots +i_k(Nk-\chi)=d\]
and hence $N(i_1+2i_2+\ldots +ki_k)=d+\chi(i_1+\ldots +i_k)$;  if this monomial appears with nonzero coefficient in $Q_i$ then  $i_1 + \ldots + i_k = i \le Nm$  and by invariance $d=0$.
Hence $Q_i$ is weighted homogeneous for the non-shifted linearisation $L^{\otimes N}$ of weighted degree $N(i_1+2i_2+\ldots +ki_k)= 
\chi  (i_1+\ldots +i_k) < 2N^2 m \le N^2 km$ as required.  

\qed

\medskip

The following classical theorem connects global invariant jet differentials to the GGL conjecture.
\begin{theorem}[Fundamental vanishing theorem, Green--Griffiths \cite{gg}, Demailly \cite{dem}, Siu \cite{siu1}]\label{demailly}
Let $N$ be the integer defined as in Remarks \ref{lineN} and \ref{lineNX}. Assume that there exist integers $k,m>0$ and ample line bundle
$A\to X$ such that there are nonzero global sections
\[0 \neq H^0(\calx_{k}^{\GIT},\calo_{\calx_{k}^{\GIT}}(m) \otimes \pi^*A^{-1})\hookrightarrow 
H^0(X,E_{k,\le N^2km} \otimes A^{-1}).\] 
Let $\s_1,\ldots ,\s_\ell$ be arbitrary nonzero sections and let $Z\subseteq J_kX$ be the base locus of
these sections. Then every entire holomorphic curve $f:\CC \to X$
necessarily satisfies $f_{[k]}(\CC)\subseteq Z$. In other words, for
every global $\Diff_k(1)$-(semi)invariant differential equation $P$ vanishing on
an ample divisor, every entire holomorphic curve $f$ must satisfy
the algebraic differential equation $P(f'(t),\ldots, f^{(k)}(t))\equiv 0$.
\end{theorem}

By Diverio \cite[Theorem 1]{div2}, for arbitrary ample $A$, $H^0(X,E_{k,m} \otimes A^{-1})= 0$
holds for all $m\ge 1$ if $k<n$, so we can restrict our attention to the range $k\ge n$. 
It turns out that $k=n$ is a good choice; in particular Theorem \ref{germtoentire} below, which is a crucial ingredient for our proof of 
 Theorem \ref{mainthmtwo}, holds when $k=n$.

To control the order of vanishing of these differential forms along
the ample divisor we will choose $A$ to be (as in \cite{dmr}) a
proper twist of the canonical bundle of $X$. Recall that the
canonical bundle of the smooth, degree $d$ hypersurface $X$ is
\[K_X=\calo_X(d-n-2),\]
which is ample as soon as $d\ge n+3$.
The following theorem summarises the results of \S 3 in Diverio--Merker--Rousseau \cite{dmr}, using the improved linear pole order for slanted vector fields due to Darondeau \cite{darondeau2}.

\begin{theorem}[Algebraic degeneracy of entire curves \cite{dmr} and \cite{darondeau2}]\label{germtoentire}
Assume that $k=n$, that $N \geq 1$ is chosen as in Remarks \ref{lineN} and \ref{lineNX}, and that  there exist  $\delta=\delta(n) >0$ and $M=M(n,\delta)$ such that
\[0 \neq H^0(\calx_{n}^{\GIT},\calo_{\calx_{n}^{\GIT}}(m) \otimes \pi^*K_X^{-2\delta Nnm})\hookrightarrow
H^0(X,E_{n,\le N^2nm} \otimes K_X^{-2\delta Nnm}) \]
whenever $\deg(X)>M(n,\delta)$ and $m \gg 0$. 
Then the Green-Griffiths-Lang conjecture holds whenever 
\[\deg(X) \ge \max(M(n,\delta), \frac{5n+3}{\delta}+n+2).\]
\end{theorem}

We will prove the following theorem.

\begin{theorem}\label{maintwo}
Let $X\subseteq \PP^{n+1}$ be a smooth complex hypersurface of degree $d$ with ample canonical bundle; that is $d \ge n+3$. Let $N \ge 1$ be as in Remarks \ref{lineN} and \ref{lineNX}. Then 
\[H^0(\calx_{n}^{\GIT},\calo_{\calx_{n}^{\GIT}}(m) \otimes \pi^*K_X^{-2\delta Nnm})\neq 0 \] 
provided that (i) $\delta=\frac{1}{16n^3}$, (ii) $\deg(X)>16n^3(5n+4)$ and (iii) $2\delta Nnm$ is an integer and is sufficiently large.
\end{theorem}
Theorem \ref{mainthmtwo} will follow from Theorem \ref{germtoentire} once we have proved Theorem \ref{maintwo}. 

To prove Theorem \ref{maintwo} we will use the
algebraic Morse inequalities of Demailly and Trapani to reduce the existence of global sections to the positivity of certain tautological integrals over $\calx_{n}^{\GIT}$.  Let $L\to P$
be a holomorphic line bundle over a compact K\"ahler manifold $P$ of
dimension $p$ and let $E \to P$ be a holomorphic vector bundle of rank $r$.
\begin{theorem}[Algebraic Morse inequalities, Demailly \cite{dem00}, Trapani \cite{trap}]\label{morse}
Suppose that $L=F\otimes G^{-1}$ is the difference of two nef line
bundles $F,G$ on $P$. Then for any nonnegative integer $q\in \ZZ_{\ge 0}$ 
\[\sum_{j=0}^q(-1)^{q-j}h^j(P,L^{\otimes m}\otimes E)\le
r \frac{m^p}{p!}\sum_{j=0}^q(-1)^{q-j}{p \choose j}F^{p-j}\cdot
G^j+o(m^p)  \mbox{ as } m \to \infty.\] 
In particular, the case when $q=1$ asserts that $L^{\otimes
m}\otimes E$ has a nonzero global section for $m$ large provided that the intersection number $F^p-pF^{p-1}G$ is strictly positive.
\end{theorem}
 In order to apply this theorem we need an expression for $\calo_{\calx_{n}^{\GIT}}(1)\otimes \pi^*K_X^{-2\d Nn}$ as a difference of nef bundles on $\calx_{n}^{\GIT}$.

\begin{proposition}\label{propnef}
Let $d\ge n+3$ so that $K_X$ ample. Let $N\ge 1$ be fixed as in Remarks \ref{lineN} and \ref{lineNX}. 
The following line bundles are nef on $\calx_{n}^{\GIT}$:
\begin{enumerate}
\item $\calo_{\calx_{n}^{\GIT}}(1) \otimes \pi^*\calo_X(2N)$
\item $\pi^*\calo_X(2N)\otimes \pi^*K_X^{2\d Nn}$ for any $\d>0$ and $2\d Nn$ integer.
\end{enumerate}
\end{proposition}

\begin{remark}
Since $K_X=\calo_X(d-n-2)$ we can write $\pi^*\calo_X(2N)\otimes \pi^*K_X^{2\d Nn}$ as $\pi^*\calo_X(2N + 2\d N n(d-n-2))$. However we prefer to keep the terms $\pi^*\calo_X(2N)$ and $ \pi^*K_X^{2\d Nn}$ separate.
\end{remark}

First note that  $T^*_{\PP^{n+1}}\otimes \calo(2)$ is globally generated, and there is a surjective bundle map $(T^*_{\PP^{n+1}}\otimes \calo(2))|_X \rightarrow T^*_X\otimes \calo_X(2)$. 
 Therefore 
\begin{equation}\label{ggen}
T^*_X\otimes \calo_X(2) \text{ is globally generated.}
\end{equation}

We now follow an inductive argument. Eliminating the terms of degree $k+1$ results in a surjective algebra homomorphism
$J_k(1,n) \twoheadrightarrow J_{k-1}(1,n))$, 
and the chain $J_k(1,n) \twoheadrightarrow J_{k-1}(1,n)
\twoheadrightarrow \ldots \twoheadrightarrow J_1(1,n)$ induces
an increasing filtration on $J_k(1,n)^*$:
\begin{equation*} 
J_1(1,n)^* \subset J_2(1,n)^* \subset \ldots \subset J_k(1,n)^*
\end{equation*}
which gives a short exact sequence of bundles
\begin{equation}\label{ses}
0 \to J_{k-1}^*X \to J_k^*X \to T_X^* \to 0
\end{equation}
This induces a short exact sequence of tangent bundles at the zero section
\begin{equation}\label{sestangent}
0 \to T_X  \to T_0(J_kX^*) \to T_0(J_{k-1}X^*) \to 0
\end{equation}
Note that the tangent bundles are vector bundles, and hence this is a short exact sequence of vector bundles over $X$. By induction and \eqref{ggen}
 we get that $T_0(J_kX^*) \otimes \calo(2k)$ is globally generated, and hence 
\begin{equation}\label{nef}
\calo_{\PP(T_0(J_k(X))}(1)\otimes \pi^*\calo_X(2)
\end{equation}
is relatively nef line bundle (note that in this paper we set $\PP(V)$ to be the bundle of 1-dimensional subspaces in $V$, not the $1$-dimensional quotients, hence sections of $\calo_{\PP(V)}(1)$ are points in $V^*$).

By Remark \ref{lineNX} the relatively ample line bundle $\calo_{\calx_k^{\GIT}}(1)$ on $\calx_k^{\GIT} \to X$, constructed as the geometric quotient of 
$$\diff_X \times_{\diff_k(n)} L^{\otimes N}|_{\tPP^{s,\hU}}$$
by the action of $\diff_k(1)$,   
pulls back on each fibre of $\calx_k^{\GIT} \to X$
 to  $L^{\otimes N}$ (or equivalently to $\calo(-NE)$ where $E$ is the exceptional divisor for the blow-up)  on $\tPP^{s,\hU} = \tPP^{0}_{\min} \setminus U Z_{\min}(\tPP)$.   
Hence \eqref{nef} tells us that $L^{\otimes N} \otimes \tilde{\pi}^*\calo(2N)$ is nef, and the induced bundle $\calo_{\calx_{n,\GL}^{\GIT}}(1) \otimes \pi^*\calo_X(2N)$ is also nef. 
The second part follows from the standard fact that the pull-back of an ample line bundle is nef.
\qed 

\medskip

Consequently, we can express $\calo_{\calx_{n}^{\GIT}}(1)\otimes \pi^*K_X^{-2\d nN}$ as the following difference of two nef line bundles:
\begin{equation*}
\calo_{\calx_{n}^{\GIT}}(1)\otimes \pi^*K_X^{-2\d nN }=(\calo_{\calx_{n}^{\GIT}}(1) \otimes
\pi^*\calo_X(2N))\otimes (\pi^*\calo_X(2N)\otimes \pi^*K_X^{2\d nN})^{-1}.
\end{equation*}
We will be able to deduce Theorem \ref{maintwo}  from the algebraic Morse inequalities Theorem \ref{morse} by proving that the top degree form $I(n,\d)$ on $\calx_{n}^{\GIT}$ given by
\begin{equation*}
I(n,\d)=
c_1(\calo_{\calx_{n}^{\GIT}}(1) \otimes \pi^*\calo_X(2N))^{n^2}-
n^2c_1(\calo_{\calx_{n}^{\GIT}}(1) \otimes
\pi^*\calo_X(2N)^{(n^2-1)}c_1(\pi^*\calo_X(2N)\otimes
\pi^*K_X^{2\d nN})
\end{equation*}
 is positive, in the sense that
$$ \int_{\calx_n^{\GIT}} I(n,\d) > 0,
$$  if $\d=\frac{1}{16n^3}$ and $d>N(n,\d)=16n^3(5n+4)$.

\section{Cohomology of $X_k^{\GIT}=\tPP/\!/\diff_k(1)$}

Recall that the fibre $X_k^{\GIT}$ of $\calx_k^{\GIT} \to X$ is the non-reductive GIT quotient $X_k^{\GIT} = \tPP/\!/ \diff_k(1)$ of the blow-up $\tPP$ of a point in the projective space
$\PP(\CC \oplus \Hom(\CC^k,\CC^n))$ by the action of $\diff_k(1) = \hU = U \rtimes \l(\CC^*)$. Our aim is to express $ \int_{\calx_n^{\GIT}} I(n,\d) $ as an integral over $X$ by representing $I(n,\d)$ using an equivariant cohomology class and 
and integrating over the fibres $X_k^{\GIT}$ of $\calx_k^{\GIT} \to X$. Recall also that since $U$ is contractible, $\hU$ is homotopy equivalent to $\CC^*$ and to $S^1$, and so $\hU$-equivariant cohomology  is isomorphic to $\CC^*$- (or equivalently $S^1$-) equivariant cohomology of $\tPP$. 

Recall also from Theorem \ref{jeffreykirwannonred}  and Remark \ref{rem:bkcoh} that we can integrate a cohomology class on the nonreductive GIT quotient $X_k^{\GIT} = \tPP/\!/\hU$ by expressing it as the image $\kappa_{\hU} (\eta)$ under the  surjective ring homomorphism $  \kappa_{\hU}: H_{\hU}^*(\tPP;\QQ)=H_{S^1}^*(\tPP;\QQ) \to H^*(\tPP/\!/\hU;\QQ)$ of some $\hU$-equivariant cohomology class $\eta$ on $\tPP$ represented by an equivariant differential form $\eta(z)$, where $z$ is the standard coordinate on the Lie algebra of $\l(\CC^*)$; we have 
\begin{equation}  \label{eqnresiduepre}  \int_{\tPP/\!/\hU} \kappa_{\hU} (\eta) 
= n_{\hU} \res_{z=0} \int_{F_{\min}}\frac{i_{F_{\min}}^* (\eta(z))} 
{\mathrm{Euler}(\mathcal{N}_{UF_{\min}})(z)} dz \end{equation}
where $F_{\min} = Z_{\min}(\tPP)$ is the union of those connected components of the fixed point locus $\tPP^{\CC^*}$ on which the $S^1$-moment map takes its minimum value $\omega_{\min}$, and $n_{\hU}$ is a strictly positive rational number  which  depends only on $\hU$ and the size of the stabiliser in $\hU$ of a generic $x \in X$. 
 Here $\mathcal{N}_{UF_{\min}}$ is the normal bundle to $UF_{\min}$ in $\tPP$, and in order to calculate the residue we express the restriction to $F_{\min} $ of its equivariant Euler class as a polynomial in $z$ with coefficients in the cohomology of $F_{\min}$, all but one of degree at least one and thus nilpotent, so that the multiplicative inverse can be expressed as a Laurent series in $z$ with coefficients in the cohomology of $F_{\min}$. This means that it is in fact more natural to write this formula as 
 \begin{equation}  \label{eqnresidue}  \int_{\tPP/\!/\hU} \kappa_{\hU} (\eta) 
= n_{\hU}\int_{F_{\min}}      \res_{z=0}    \frac{i_{F_{\min}}^* (\eta(z))} 
{\mathrm{Euler}(\mathcal{N}_{UF_{\min}})(z)} dz. \end{equation}
 Moreover as noted in Remark \ref{rem:bkcoh}, it suffices to express our cohomology class on $\tPP/\!/\hU$ as 
 $$ \kappa_{\hU}|_{\tilde{J}_k(1,n)}(\eta)$$
 where $\tilde{J}_k(1,n)$ is the blow-up of $J_k(1,n)$ at 0 and 
 $$ \eta \in H^*_{\hU}(\tilde{J}_k(1,n);\QQ) \cong H^*_{S^1}(\tilde{J}_k(1,n);\QQ).$$
 Thus to understand the cohomology of $X_k^{\GIT}$ we want to consider first the equivariant cohomology of 
$Z_{\min}(\tPP) \subseteq  \tilde{J}_k(1,n) \subseteq \tPP$ with respect to the action of $\diff_k(1) = \hU = U \rtimes \l(\CC^*)$. 

Since $J_k(1,n)$ is an affine space, $\tilde{J}_k(1,n)$ retracts equivariantly onto its exceptional divisor $E \cong \PP^{kn-1}$, which is a GIT quotient of $T_0 J_k(1,n)$ by the $\CC^*$-action of scalar multiplication which commutes with the action of $\hU$. Let $T_2$ be the two-dimensional complex torus which is the product of this $\CC^*$ acting via scalar multiplication and the one-parameter subgroup $\l:\CC^* \to \hU$ of $\hU$; its Lie algebra has coordinates $(w,z)$ where $w$  is the standard coordinate on the Lie algebra of the copy of $\CC^*$ acting via scalar multiplication and $z$  is the standard coordinate on the Lie algebra of $\l(\CC^*)$.
The $T_2$-weights on the tangent space $T_0 J_k(1,n) = T_{[1:0:\ldots :0]} \PP(\CC \oplus \Hom(\CC^k, \CC^n))$ are $w+z,w+2z,\ldots, w+kz$ (all of these occurring with multiplicity $n$). 

Embedding $\tPP$ in $\PP^{kn} \times \PP^{kn-1}$,  let $\zeta$ denote the equivariant first Chern class of the hyperplane line bundle on $P^{kn}$ (which restricts to the trivial bundle on $J_k(1,n)$), and let $\eta$ denote the equivariant first Chern class of the hyperplane line bundle on $\PP^{kn-1} = \PP(T_0 J_k(1,n))$,  with both pulled back to  equivariant classes on $\tPP$. To apply Theorem \ref{mainthm} we take $a>>1$ and twist the restriction to $\tPP$ of the ample linearisation $L=\calo_{\PP^{kn}}(a) \otimes \calo_{\PP^{kn-1}}(1)$ on $\PP^{kn} \times \PP^{kn-1}$ by a rational character of $\hU$ which can be identified with $bz$ for some $b \in \QQ$ and makes the linearisation well-adapted for the action of $\hU$. Then $\tZmin = Z_{\min}(\tPP)$ is the $n-1$-dimensional projective linear subspace of $E = \PP(T_0J_k(1,n) \cong \PP^{kn-1}$ corresponding to the $T_2$-weight space in $T_0 J_k(1,n)$ with weight $w+z$, and to make the linearisation well-adapted for the $\hU$-action we twist by a rational character of $\hU$ (and thus of $T_2$) so that this weight $w+z$ becomes $w-\epsilon z$ for some  small positive $\epsilon$. So the restriction of the well-adapted $\hU$-equivariant linearisation on $\tPP$ to $\tZmin = Z_{\min}(\tPP)$ can be identified (up to replacement with a positive tensor power) with the ample $\hU$-equivariant line bundle (defined up to replacement with a positive tensor power) induced on $Z_{\min}(\tPP)$ by regarding $Z_{\min}(\tPP)$ as the GIT quotient by $\CC^*$ of the $n$-dimensional weight space in $T_0 J_k(1,n)$ with $\CC^* \times \hU$-weight $w -\epsilon z$ where $0 < \epsilon << 1$. In fact 
in \S5.2 of \cite{bkgrosshans} it is  shown that, because the action of $\hU$ on $\tPP$ extends to an action of $\GL(k)$, we can choose any rational $0<\epsilon<1$; we will make a choice  of $\epsilon$ in this range later.

Finally recall that the non-reductive GIT quotient $X_k^{\GIT}=\tPP/\hU = \tPP/\!/\Diff_k$ given by Theorem \ref{mainthm}  is a smooth projective variety of dimension $\dim(X_k^{\GIT})=k(n-1)$, and 
a sufficiently divisible power  
\begin{equation}\label{defN2}
L^{\otimes N}
\end{equation}
of $L$ induces an ample line bundle $\calo_{X_k^{\GIT}}(1)$ on this quotient; this line bundle pulls back to the restriction of $L^{\otimes N}$ to the (semi)stable locus in $\tPP$ (see Remark \ref{explainN}). The first Chern class of $\calo_{X_k^{\GIT}}(1)$ has an equivariant lift to $\tPP$ which is the first Chern class of the well-adapted twist of $L^{\otimes N}$, and by choosing $N$ to be sufficiently divisible we can assume that  its restriction to $\tZmin = Z_{\min}(\tPP)$ is  the ample $\hU$-equivariant line bundle  induced on $Z_{\min}(\tPP)$ by regarding $Z_{\min}(\tPP)$ as the GIT quotient by $\CC^*$ of the $n$-dimensional weight space $T_0 J_k(1,n)_{w+z}$ in $T_0 J_k(1,n)$ with $\CC^* \times \hU$-weight $N(w -\epsilon z)$.

Our aim is to show that the integral 
$ \int_{\calx_n^{\GIT}} I(n,\d) $
is strictly positive by integrating over the fibres of $\calx_n^{\GIT} \to X$ to obtain an integral over $X$. 
We will find that in order to integrate over the fibres
we will need to calculate 
 \[\int_{X_k^{\GIT}}  c_1(\calo_{X_k^{\GIT}}(1))^{k(n-1)} \]
 when $k=n$, as well as some other similar integrals. It follows from (\ref{eqnresidue}) that 
 $$ \int_{X_k^{\GIT}}  c_1(\calo_{X_k^{\GIT}}(1))^{k(n-1)}
 = n_{\hU}  \int_{\tZmin}    \res_{z=0}   \frac{i_{F_{\min}}^* c_1(L^{\otimes N})(z)} 
{\mathrm{Euler}(\mathcal{N}_{U\tZmin})(z)} dz. $$
By applying (\ref{eqnresidue}) in the case when $U$ is trivial and $F_{\min} = \{0\} \subseteq T_0 J_k(1,n)_{w+z} \subseteq \PP(\CC \oplus T_0 J_k(1,n))$, we have 
$$\int_{\tZmin}\frac{i_{F_{\min}}^* c_1(L^{\otimes N})(z)} 
{\mathrm{Euler}(\mathcal{N}_{U\tZmin})(z)}  = n_{\CC^*} \res_{w=0} \frac{\gamma(w,z)} 
{\mathrm{Euler}(\mathcal{N}_{T_0 J_k(1,n)_{w+z}
})(w,z)} dw$$
where $\gamma(w,z)$ is a $T_2$-equivariant cohomology class on $T_0 J_k(1,n)_{w+z}$
which represents the restriction of the equivariant first Chern class of the well-adapted twist of $L^{\otimes N}$ to $Z_{\min}(\tPP) = \PP(T_0 J_k(1,n)_{w+z})$  quotiented by 
the equivariant Euler class $\mathrm{Euler}(\mathcal{N}_{U\tZmin})(z)$, and $\mathcal{N}_{T_0 J_k(1,n)_{w+z}
}$ is the normal bundle to ${T_0 J_k(1,n)_{w+z}
}$ in $T_0J_k(1,n)$. Note that the normal bundle $\mathcal{N}_{U\tZmin}$ to $U \tZmin$ in $\tPP$ is the quotient of the normal bundle $\mathcal{N}_{\tZmin}$ to $\tZmin$ by the trivial bundle $V_\lieu$ with fibre $\lieu$.

Thus we will obtain an expression for $ \int_{X_k^{\GIT}}  c_1(\calo_{X_k^{\GIT}}(1))^{k(n-1)}$ as an iterated residue of a rational function of $w$ and $z$.

\subsection{Integration on $X_k^{\GIT}$}
Let us collect the ingredients needed to calculate $  \int_{X_k^{\GIT}} \,\, c_1(\calo_{X_k^{\GIT}}(1))^{k(n-1)} $ using 
 Theorem \ref{jeffreykirwannonred} and Remark \ref{rem:bkcoh} as above. We have fixed a linearisation for the action of $\hU = \diff_k(1)$ on $\tPP$ with induced ample line bundle $\calo_{X_k^{\GIT}}(1)$
on $X_k^{\GIT}={\tPP/\!/\diff_k(1)}$ (where $\tPP$ is  the projective space $\PP(\CC \oplus \Hom(\CC^k, \CC^n))$ blown up at $[1:0 \cdots 0]$), and 
a two-dimensional complex torus $T_2 = \CC^* \times \l(\CC^*)$ with the first copy of $\CC^*$ acting on the tangent space $T_0 J_k(1,n)$ as scalar multiplication, and $\l(\CC^*)$ acting on $T_0 J_k(1,n)$ as the restriction of the induced action of $\hU$.
We have seen 
that 
 $$ \int_{X_k^{\GIT}}  c_1(\calo_{X_k^{\GIT}}(1))^{k(n-1)}
 = n_{\hU} \int_{\tZmin}    \res_{z=0}    \frac{i_{\tZmin}^* c_1(L^{\otimes N})(z) \cup \mathrm{Euler}(V^*_\lieu)(z)}
{\mathrm{Euler}(\mathcal{N}_{\tZmin})(z)} dz $$
and
$$\int_{\tZmin}\frac{i_{\tZmin}^* c_1(L^{\otimes N})(z) \cup \mathrm{Euler}(V^*_\lieu)(z)}
{\mathrm{Euler}(\mathcal{N}_{\tZmin})(z)}  = n_{\CC^*} \res_{w=0} \frac{\gamma(w,z)} 
{\mathrm{Euler}(\mathcal{N}_{T_0 J_k(1,n)_{w+z}
})(w,z)} dw$$
where $\gamma(w,z)$ is a $T_2$-equivariant cohomology class on $T_0 J_k(1,n)_{w+z}$
which represents the restriction of the equivariant first Chern class of the well-adapted twist of $L^{\otimes N}$ to $\tZmin = \PP(T_0 J_k(1,n)_{w+z})$ quotiented by 
$$\mathrm{Euler}(\mathcal{N}_{U\tZmin})(z) = \mathrm{Euler}(\mathcal{N}_{\tZmin})(z)/ \mathrm{Euler}(V^*_\lieu)(z) ,$$
and $\mathcal{N}_{T_0 J_k(1,n)_{w+z}
}$ is the normal bundle to ${T_0 J_k(1,n)_{w+z}
}$ in $T_0J_k(1,n)$.
 Here $w,z$ are the standard coordinates on the Lie algebra of $T_2$. Substituting we obtain
\begin{equation} \label{doubleres}
\int_{X_k^{\GIT}}  c_1(\calo_{X_k^{\GIT}}(1))^{k(n-1)}
 = n_{\hU \times \CC^*} \res_{w=0}  \res_{z=0} \frac{\gamma(w,z) \cup \mathrm{Euler}(V^*_\lieu)(w,z)}
{\mathrm{Euler}(T_0 J_k(1,n)_{w+z}
)(w,z)} dw dz.
\end{equation}
The ingredients appearing in this equation can be summarised as follows. 
\begin{enumerate}
\item $n_{\hU \times \CC^*}$ is a strictly positive rational number which depends on ${\hU \times \CC^*}$ and the size of a generic stabiliser for its action on $T_0 J_k(1,n)$.
\item The restriction of the equivariant first Chern class of the well-adapted twist of $L^{\otimes N}$ to $\tZmin$ is represented by the $T_2$-equivariant Euler class $Nw - N\epsilon z$ on the tangent space $T_0J_k(1,n)$.
\item The weights of $\l: \CC^* \to \hU = \diff_k(1)$ on $\lieu=\mathrm{Lie}(U)$ are $1,2,\ldots, (k-1)$, and the $\l(\CC^*)$-equivariant Euler class $\mathrm{Euler}(V_\lieu^*)$ is represented by the $T_2$-equivariant class $
-z\cdot (-2z) \cdot \ldots \cdot (-(k-1)z)=(k-1)!(-z)^{k-1}$ on $T_0 J_k(1,n)$.
\item The normal bundle  to $\tZmin$ in $E$ has equivariant Euler class represented by 
$(-z)^n(-2z)^n \cdot \ldots \cdot (-(k-1)z)^n$. The weights here correspond to the second, third, ..., and last columns of $E\cong \PP(T_0J_k(1,n))$. The normal bundle of $E$ in $\tPP$ is the tautological bundle $\calo_E(-1)$ with equivariant first Chern class represented by $-w-z$, and hence 
the equivariant Euler class $\mathrm{Euler}(\mathcal{N}_{\tZmin})$ of the normal bundle $\mathcal{N}_{\tZmin}$ to $\tZmin$ in $\tPP$ is represented by 
\[(-1)^{(k-1)n+1}(w+z)((k-1)!)^{n}z^{(k-1)n}\] 
\item The equivariant Euler class $\mathrm{Euler}(T_0 J_k(1,n)_{w+z})$ is represented by $(w+z)^n$.
\end{enumerate} 
Substituting this into the formula (\ref{doubleres}), and changing $z$ to $-z$, which changes the sign of the iterated residue, 
we obtain the following result.
\begin{proposition}\label{intonfibers2} With the notation introduced above we have 
\[\int_{X_k^{\GIT}}c_1(\calo_{X_k^{\GIT}}(1))^{k(n-1)}= 
n_{\hU \times \CC^*} \res_{w=0}  \res_{z=0} 
 \frac{(Nw+N\epsilon z)^{k(n-1)}dzdw}{(w-z)^{n+1}((k-1)!)^{n-1}z^{(k-1)(n-1)}}
\]
where $n_{\hU \times \CC^*}$ is a strictly positive rational, and  the iterated residue means expansion on a contour where $|w|\gg |z|$ and taking the coefficient of $(zw)^{-1}$.
\end{proposition}
Note that the rational expression on the right hand side has positive expansion on the contour $|w|\gg |z|$, that is, as a Laurent series in $\frac{z}{w}$. In particular we see that the residue, which is the coefficient of $(zw)^{-1}$, is positive, as it must be, because the degree of an ample line bundle is positive.
\section{Integration on $
\calx_k^{\GIT}
=
\widetilde{J_k}X/\!/\diff_k(1)$}\label{section:proofmaintechnical}

To evaluate the integral $\int_{\calx_{k}^{\GIT}}\calo_{\calx_{k}^{\GIT}}(1)^{n+k(n-1)}$ (or  the integral over $\calx_k^{\GIT}$ of the product of $c_1(\calo_{\calx_k^{\GIT}}(1))^{n+k(n-1)-j}$ and the pullback to $\calx_k^{\GIT}$ of a cohomology class $\zeta$ on $X$ of degree $j$),
we can first integrate (push forward) along the fibres of $\pi : \calx_{k}^{\GIT} \to X$ and then integrate over $X$. Integration along the fibres can be done using the iterated residue formula of Proposition \ref{intonfibers2}. 
Here we must replace $\tZmin$ with 
$$\diff_X \times_{\diff_k(n)} \tZmin,$$
which we can identify with the projectivised tangent bundle 
$\PP(TX)$ for $X$.
In the relative version of Proposition \ref{intonfibers2} the $T_2$-equivariant Euler class of the tangent space $T_0 J_k(1,n)_{w+z}$ is replaced with the equivariant Euler class of this tangent bundle. This is given by 
\[\prod_{i=1}^n(\lambda_i+w + z),\]\
where $\l_1,\ldots, \l_n$ are the Chern roots of $TX$. 
Similarly the $T_2$-equivariant class 
 representing the 
equivariant Euler class $\mathrm{Euler}(\mathcal{N}_{\tZmin})$ of the normal bundle $\mathcal{N}_{\tZmin}$ to $\tZmin$  is replaced with
$$ -(w+ z) \prod_{l=1}^{k-1}\prod_{i=1}^n(\l_i-lz).$$
Hence after substituting $-z$ for $z$ as before, so that the iterated residue changes sign, we find that the relative version of Proposition \ref{intonfibers2} is as follows (with a similar formula for the integral over $\calx_k^{\GIT}$ of the product of $c_1(\calo_{\calx_k^{\GIT}}(1))^{n+k(n-1)-j}$ and the pullback to $\calx_k^{\GIT}$ of a cohomology class $\zeta$ on $X$ of degree $j$).

\begin{proposition}\label{intonfibers3} 
\[\int_{\calx_k^{\GIT}}c_1(\calo_{\calx_k^{\GIT}}(1))^{n+k(n-1)}= n_{\hU \times \CC^*} \res_{w=0}  \res_{z=0} 
 \frac{(k-1)!z^{k-1}(Nw+ N\epsilon z)^{n+k(n-1)}dzdw}{(w-z)\prod_{i=1}^n(\lambda_i+w-z)\prod_{l=1}^{k-1}\prod_{i=1}^n(\l_i+lz)}.\]
 Here $n_{\hU \times \CC^*}$ is a strictly positive rational, and the iterated residue means expansion on a contour where $|z|\gg |w|$, and taking the coefficient of $(zw)^{-1}$.
\end{proposition}
 
The inverse of the total Chern class $c_X$ of $TX$ is the total Segre class $S(X)$ of the tangent bundle, so $c(X)^{-1}=s(X)$ and we can write  
\[\frac{1}{\prod_{i=1}^n((\l_i+x)}=\frac{1}{x^n}s\left(\frac{1}{x}\right) \text{ for } x=w-z,\ lz\]
Hence the relative version of Proposition \ref{intonfibers2} gives us the following integration formula on $\calx_k^{\GIT}$. 
\begin{theorem}\label{intonfibersb} The integral over $\calx_k^{\GIT}$ of the product of $c_1(\calo_{\calx_k^{\GIT}}(1))^{n+k(n-1)-j}$ and the pullback to $\calx_k^{\GIT}$ of a cohomology class $\zeta$ on $X$ of degree $j$ is given by 
\begin{equation} \label{intonfibersb2}
n_{\hU \times \CC^*} \int_X  \zeta \, \res_{w=0} \res_{z=0}  \frac{(Nw+N\epsilon z)^{n+k(n-1)-j}dwdz}{((k-1)!)^{n-1}(w-z)^{n+1}z^{(k-1)(n-1)}}s\left(\frac{1}{w-z}\right)\prod_{l=1}^{k-1}s\left(\frac{1}{lz}\right).
\end{equation}
Here $n_{\hU \times \CC^*}$ is a strictly positive rational, and the residue is a homogeneous  polynomial of degree $n=\dim(X)$ in the Segre classes of $X$, which is then integrated over $X$.
\end{theorem}

\begin{remark}\label{explanation} There are three key features of this formula which makes intersection calculation significantly easier:
\begin{enumerate}
\item It separates the Segre classes from the residue variables, and hence the residue is a polynomial of degree $n$ in the Segre classes. 
\item The rational expression has positive expansion in $\frac{z}{w}$, that is, all coefficients in the expansion on the contour $|w|\gg |z|$ are positive.
\item Although $N$ plays a crucial role in finding the relative NEF bundles in the Morse inequalities, it only appears as a factor in the formula of Theorem \ref{intonfibersb}, and hence the positivity of the integral does not depend on the actual value of $N$. This is a crucial aspect of our integral formula, because we do not know the exact value of $N$: it only guarantees the relative ampleness of the bundle $L$. 
\end{enumerate}
\end{remark}

\section{Proof of Theorem \ref{mainthmtwo}}\label{sec:proof}

In this section we complete the proof of Theorem \ref{maintwo}.

Theorem \ref{maintwo} will follow from the Morse inequalities by proving that the following top degree form on $\calx_{n}^{\GIT}$ is positive if $\d=\frac{1}{16n^5}$ and $d>N(n,\d)=2(4n)^5$:
\begin{equation}\label{intnumbersnowman}
I(n,\d)=
c_1(\calo_{\calx_{n}^{\GIT}}(1) \otimes \pi^*\calo_X(2N))^{n^2}-
n^2c_1(\calo_{\calx_{n}^{\GIT}}(1) \otimes
\pi^*\calo_X(2N)^{(n^2-1)}c_1(\pi^*\calo_X(2N)\otimes
\pi^*K_X^{2\d nN}).
\end{equation}
Recall the notation $h=c_1(\calo_X(1)), u=c_1(\calo_{\calx_{k,\GL}^{\GIT}}(1))$, and $c_1=c_1(T_X)$ for the corresponding first Chern classes. Then $c_1(K_X)=-c_1=(d-n-2)h$, and by dropping $\pi^*$ from our formula \eqref{intnumbersnowman} can be rewritten as
\begin{equation}\label{rform0}
I(n,\d)=(u+2Nh)^{n^2}-
2n^2(u+2Nh)^{n^2-1}(Nh+\d nN(d-n-2)h).
\end{equation}
In the residue formula of Theorem \ref{intonfibersb} we substitute $u=Nw+N\epsilon z$ and we obtain 
\begin{equation}\label{rform}
I_{n,\d}(z,w,h)=(Nw+N \epsilon z+2Nh)^{n^2}-
2n^2(Nw+N \epsilon z+2Nh)^{n^2-1}(Nh+\d nN(d-n-2)h).
\end{equation}

Since $X\subseteq \PP^{n+1}$ is a projective hypersurface we can express the Segre classes in Theorem \ref{intonfibersb} using that the Chern classes of $X$ are expressible with $d=\deg(X)$ and $h$:
\[(1+h)^{n+2}=(1+dh)c(X),\]
where $c(X)=c(T_X)$ is the total Chern class of $X$. This gives 
\[s\left(\frac{1}{x}\right)=\frac{1}{c(1/x)}=\left(1+\frac{dh}{x}\right)\left(1-\frac{h}{x}+\frac{h^2}{x^2}-\ldots \right)^{n+2}.\]
\begin{proposition}\label{propresiduethree} The integral $\int_{\calx_n^{\GIT}} I_{n,\d}$ is given by $n_{\hU \times \CC^*}$ times 
\[\int_X  \, \res_{w=0} \res_{z=0}  \frac{I_{n,\d}(z,w,h)dwdz}{((n-1)!)^{n-1}(w-z)^{n+1}z^{(n-1)^2}}\left(1+\frac{dh}{w-z}\right)\left(1-\frac{h}{w-z}+\ldots \right)^{n+2}\prod_{l=1}^{n-1}\left(1+\frac{dh}{lz}\right)\left(1-\frac{h}{lz}+\ldots \right)^{n+2}\]
where
\[
I_{n,\d}(z,w,h)=N^{n^2}\left((w+ \epsilon z+2h)^{n^2}-
2n^2h(w+\epsilon z+2h)^{n^2-1}(1+\d n(d-n-2))\right).
\]
and 
\[\frac{1}{w-z}=\frac{1}{w}\left(1+\frac{z}{w}+\frac{z^2}{w^2}+\ldots \right).\]
\end{proposition}

\subsection{Residue integral formula on $X_2^{\GIT}$ 
} \label{sec:resexamoples}

Before we begin the analysis of the integral 
$\int_{\calx_n^{\GIT}} I_{n,\d}$, we will study it in the case $k=n=2$, building on the description of  $X_2^{\GIT}$  in \S \ref{subsection:example}.

For $k=n=2$ we have $n_{\hU \times \CC^*}=1$ and Proposition \ref{propresiduethree} gives
\[\int_{\calx_2^{\GIT}}I_{2,\d}=\int_X \res_{w=0} \res_{z=0} \frac{I_{2,\d}(z,w,h)}{w^3z}\left(1+\frac{z}{w}+\ldots \right)^3 \left(1+\frac{dh}{w-z}\right)\left(1-\frac{h}{w-z}+\ldots \right)^{4} 
\left(1+\frac{dh}{z}\right)\left(1-\frac{h}{z}+\ldots \right)^{4}.
\]
where (after dropping the $N^4$ factor)
\[
I_{2,\d}(z,w,h)=(w+\epsilon z+2h)^3 \cdot \left(w+\epsilon z-16\d dh-\left(6-64\d\right)h\right).
\]
The iterated residue will not involve any power $h^t$ for $t>2$, and moreover, we cannot take terms where the degree of $w$ in the denominator is $6$ or larger. This leads to significant simplification:
\begin{align*}
& \coeff_{(zw)^{-1}} \frac{I_{2,\d}(z,w,h)}{w^3z}\left(1+\frac{z}{w}+\ldots \right)^3 \left(1+\frac{dh}{w-z}\right)\left(1-\frac{h}{w-z}+\frac{h^2}{(w-z)^2} \right)^{4} 
\left(1+\frac{dh}{z}\right)\left(1-\frac{h}{z}+\frac{h^2}{z^2} \right)^{4}=\\
& \coeff_{(zw)^{-1}} \frac{I_{2,\d}(z,w,h)}{w^3z}\left(1+\frac{z}{w}+\ldots \right)^3 \left(1+\frac{dh}{w-z}\right)\left(1-\frac{4h}{w-z}+\frac{10h^2}{(w-z)^2} \right)
\left(1+\frac{dh}{z}\right)\left(1-\frac{4h}{z}+\frac{10h^2}{z^2} \right)=\\\
& \coeff_{(zw)^{-1}} \frac{I_{2,\d}(z,w,h)}{w^3z}\left(1+\frac{3z}{w}+\frac{6z^2}{w^2} \right) \left(1+\frac{(d-4)h}{w-z}+\frac{(10-4d)h^2}{(w-z)^2}\right)\left(1+\frac{(d-4)h}{z}+\frac{(10-4d)h^2}{z^2}\right)=\\
& \coeff_{(zw)^{-1}} \frac{I_{2,\d}(z,w,h)}{w^3z}\left(1+\frac{3z}{w}+\frac{6z^2}{w^2} \right) \left(1+\frac{(d-4)h}{w}\left(1+\frac{z}{w}\right)+\frac{(10-4d)h^2}{w^2}\right)\left(1+\frac{(d-4)h}{z}+\frac{(10-4d)h^2}{z^2}\right)=\\
& 4d^3(1+\epsilon-16\delta-12\delta \epsilon)+4d^2(104\delta+96\delta \epsilon-15-6\epsilon^2-20\epsilon)+2d(55+92\epsilon+30\epsilon^2-320\delta-384\delta \epsilon)
\end{align*}
The leading coefficient must be positive, which forces the inequality
\[\d < \frac{1+\epsilon}{16+12\epsilon}\]
and by the condition $0< \epsilon <1$ this results in $\delta< 1/14$, hence the first bound in Theorem \ref{germtoentire} for the degree satisfies 
\[\frac{5n+3}{\delta}+n+2 \ge 186,\]
and we can not expect better degree bound with our approach. By choosing $\epsilon=1/4$, as in the general argument in the next sextion, the polynomial has the form
\[\int_{\calx_2^{\GIT}}I_{2,\d}=4d^3(5/4-19\d)-4d^2(\frac{163}{8}-128\d)+2d(78+\frac{15}{8}-416 \d)\]
We pick $\delta=\frac{1}{16}$ to ensure the positivity of the leading coefficient. Note that in the main theorem our choice is $\delta=\frac{1}{16n^3}=\frac{1}{128}$, which is slightly smaller; this is due to estimation of the leading coefficient. Then the integral is positive for $d>99+\sqrt{9370} \sim 200$. Hence by Theorem \ref{germtoentire} the Green-Griffiths-Lang conjecture holds for 
\[\deg(X) \ge \max(190,\frac{5n+3}{\delta}+n+2)=212.\]
With fine-tuning the parameters we can get slighlty closer to the optimal bound $d=186$ mentioned above.
\begin{remark} Our approach can not compete with the best known degree bounds for small $n$. In particular, for projective surfaces $X \subset \PP^3$ the best known degree bound is $\deg(X)=18$  (see section 5.4 in \cite{dr}), and we do not have a well-established explanation for the gap between our result and the best bound. Intuitively, due to well-adaptedness and the ample twist determined by $\epsilon$ and $N$, our non-reductive GIT model does sees only a graded sub-algebra of the invariant jet-differentials.    
\end{remark}

\subsection{A first look at the iterated residue formula}\label{subsec:residue}

As a first step in the analysis of the residue formula of Proposition \ref{propresiduethree} we write this integral as a polynomial in $d$ and study its leading coefficient. 
For a nonnegative integer $i$ and a partition $i=i_0+i_1+\ldots +i_{n-1}$ into integer vector $\bi=(i_0,\ldots, i_{n-1})$ we introduce the shorthand notation 
\[C^{\bi}=\res_{w=0} \res_{z=0} \frac{(w+\epsilon z)^{n^2-i}dwdz}{(w-z)^{n+2-i_0}\prod_{l=1}^{n-1}(lz)^{n-i_l}}.\]
These will be the building blocks of our integral formula and we can explicitly calculate this residue using the expansion 
\begin{equation}\label{cbi}
C^{\bi}=\frac{\prod_{l=1}^{n-1}l^{i_l}}{((n-1)!)^n} \mathrm{coeff}_{(zw)^{-1}} \frac{(w+\epsilon z)^{n^2-(i_0+\ldots +i_{n-1})}}{w^{n+2-i_0}z^{(n-1)n-(i_1+\ldots +i_{n-1})}}\left(1+\frac{z}{w}+\frac{z^2}{w^2} +\ldots \right)^{n+2-i_0}
\end{equation}
Note that all coefficients of this expansion are positive. 

\begin{proposition}\label{leadingcoeff}
\begin{enumerate}
\item
$\int_{\calx_n^{\GIT}} I_{n,\d}$ is a polynomial in $d$ of degree $n+1$ with zero constant term:
\[\int_{\calx_n^{\GIT}} I_{n,\d}=p_{n+1}(n,\d)d^{n+1}+p_{n}(n,\d)d^{n}+\ldots +p_1(n,\d)d\]
where $p_i(n,\d)$ is linear in $\d$ and polynomial in $n$ for all $i$.
\item The leading coefficient is $p_{n+1}(n,\d)>C^{\mathbf{0}}\left(1-\frac{\d n^3}{\epsilon}\right)$ and $C^{\mathbf{0}}=C^{\mathbf{0}}(N,n,\epsilon)>0$ is positive.
\end{enumerate}
\end{proposition}
\proof
The residue in Proposition \ref{propresiduethree} is by definition the coefficient of $\frac{1}{zw}$ in the Laurent expansion of the rational expression in $z, n, d, h$ and $\d$ on the contour $|w|\gg |z|$, that is, in $z/w$.  The result is a polynomial in $n,d,h,\d$, and in fact, a relatively easy argument shows that it is a polynomial in $n,d,\d$ multiplied by $h^n$ Indeed, setting degree $1$ to $z,w, h$ and $0$ to $n, d, \d$, the rational expression in the residue has total degree $n-2$. Therefore the coefficient of $\frac{1}{zw}$ has degree $n$, so it has the form $h^n p(n,d,\d)$ with a polynomial $p$. Since $\int_X h^n=d$, integration over $X$ is simply a substitution $h^n=d$, resulting in the equation $\int_{\calx_n^{\GIT}} I_{n,\d}=dp(n,\d,d)$ for some polynomial $p(n,\delta,d)$. The highest power of $d$ in $p(n,\delta,d)$ is $d^n$ which proves the first part.

To prove the second part note that to get $d^{n+1}$ in Proposition \ref{propresiduethree} we have two options. We drop the $N^{n^2}$ factor from the formulas below. 

(i) The first is to choose the $\frac{dh}{w-z}$ and $\frac{dh}{lz}$ terms in the product $\left(1+\frac{dh}{w-z}\right) \prod_{l=1}^{n}\left(1+\frac{dh}{lz} \right)$, this contributes with  
\[C^{\mathbf{0}}=\frac{1}{((n-1)!)^n} \mathrm{coeff}_{(zw)^{-1}} \frac{(w+\epsilon z)^{n^2}}{w^{n+2}z^{(n-1)n}}\left(1+\frac{z}{w}+\frac{z^2}{w^2} +\ldots \right)^{n+2}=\frac{1}{((n-1)!)^n}\sum_{i=0}^{n(n-1)-1} {n^2 \choose i}\epsilon^i \Gamma^{(n+2)}_{n(n-1)-1-i},\]
where $\Gamma^{(n+2)}_j=\mathrm{coeff}_{(z/w)^j} \left(1+\frac{z}{w}+\frac{z^2}{w^2} +\ldots \right)^{n+2}$.

(ii) Alternatively we can pick all but one $dh$ terms from the product $\left(1+\frac{dh}{w-z}\right) \prod_{l=1}^{n}\left(1+\frac{dh}{lz} \right)$ and the $2\d n^3dh$ term from $I_{n,\d}(z,w,h)$. This way the contribution is -$\sum_{s=0}^{n-1}2\d n^3 C^{\mathbf{e_s}}$ where $\mathbf{e}_s$ is the unit vector with all but the $s$ coordinate zero. From \eqref{cbi} these terms are
{\small \[C^{\mathbf{e}_s}=\begin{cases} \frac{s}{((n-1)!)^n} \mathrm{coeff}_{(zw)^{-1}} \frac{(w+\epsilon z)^{n^2-1}}{w^{n+2}z^{(n-1)n-1}}\left(1+\frac{z}{w}+\frac{z^2}{w^2} +\ldots \right)^{n+2}=\frac{s}{((n-1)!)^n}\sum_{i=0}^{n(n-1)-2} {n^2-1 \choose i}\epsilon^i \Gamma^{(n+2)}_{n(n-1)-2-i} &  1\le s \le n-1 \\ 
\frac{1}{((n-1)!)^n} \mathrm{coeff}_{(zw)^{-1}} \frac{(w+\epsilon z)^{n^2-1}}{w^{n+1}z^{(n-1)n}}\left(1+\frac{z}{w}+\frac{z^2}{w^2} +\ldots \right)^{n+1}=\frac{1}{((n-1)!)^n}\sum_{i=0}^{n(n-1)-1} {n^2-1 \choose i}\epsilon^i \Gamma^{(n+1)}_{n(n-1)-1-i} &  s=0 \end{cases}.\]}
All coefficients are positive, and by a) comparing the $i$th term in $C^{\mathbf{0}}$ with the $i-1$th term in $C^{\mathbf{e}_s}$ when $1\le s \le n-1$ for $i\ge 1$ and b) comparing the $i$th term in $C^{\mathbf{0}}$ with the $i$h term in $C^{\mathbf{e}_s}$ when $s=0$, we obtain 
\begin{equation}\label{compare}
0<C^{\mathbf{e_s}}< \begin{cases} \frac{sC^{\mathbf{0}}}{2n^2\epsilon} & \text{ if } 1\le s \le n-1 \\ C^{\mathbf{0}} & \text{ if } s=0 \end{cases}
\end{equation}
 holds and hence the total contribution is less than $2\d n^3 C^{\mathbf{0}}(1+\frac{1}{4\epsilon})<\frac{\d n^3 C^{\mathbf{0}}}{\epsilon}$ if $0< \epsilon<1$. Thus the second part is proved. 
\qed

We obtain the following corollary. 
\begin{proposition}\label{posleading}
If $\delta < \frac{\epsilon}{n^3}$ then the  leading coefficient $p_{n+1}(n,\delta)>0$ is positive, and therefore $\int_{\calx_n^{\GIT}} I_{n,\d}>0$ for $d\gg 0$.  
\end{proposition}
Note that we will choose $\epsilon =1/2$ after calibrating the parameters. 

\subsection{Proof of positivity of $I(n,\d,d)$}\label{subsec:coeffs}  According to Proposition \ref{leadingcoeff} we have to prove the positivity of the polynomial $\int_{\calx_n^{\GIT}}=p_{n+1}(n,\d)d^{n+1}+p_{n}(n,\d)d^{n}+\ldots +p_1(n,\d)d$. The strategy is to show that for small enough $\delta$ the other coefficients satisfy 
\begin{equation}\label{goal}
|p_{n+1-l}|<(12n)^{4l}p_{n+1}
\end{equation}
for $1\le l \le n+1$. Then we can apply the following elementary statement.
\begin{lemma}[Fujiwara bound]\label{estimation}
If $p(d)=p_{n+1}d^{n+1}+p_{n}d^{n}+\ldots +p_1d+p_0\in \RR[d]$ satisfies the inequalities 
\[p_{n+1}>0;\ \ |p_{n+1-l}|<D^l |p_{n+1}| \text{ for } l=1,\ldots n+1,\]
then $p(d)>0$ for $d>2D$.  
\end{lemma}

We start with the study of the next coefficient, $p_n$. Similarly to the proof of Proposition \ref{leadingcoeff} (2), here we can distinguish four cases how we can get $d^n$ in the residue formula of Proposition \ref{propresiduethree}.

(i) If we take $n-1$ $dh$ terms from the product $\left(1+\frac{dh}{w-z}\right) \prod_{l=1}^{n-1}\left(1+\frac{dh}{lz} \right)$ and one $h$ from $I_{n,\d}(z,w,h)$ then we get
\[A=\sum_{s=0}^{n-1}2\d n^3(n+2) C^{\mathbf{e}_s}.\]

(ii) If we take $n-1$ $dh$ terms from $\left(1+\frac{dh}{w-z}\right) \prod_{l=1}^{n-1}\left(1+\frac{dh}{lz} \right)$ (namely, we drop the $s$th $dh$ term) and one $h$ from $(1-\frac{h}{w-z}+\ldots)^{n+2}\prod_{l=1}^{n-1}(1-\frac{h}{lz}+\ldots)^{n+2}$ (namely, the $h$ in the $t$th term) then the contribution is 
\[-B=-\sum_{s=0}^{n-1}\sum_{t=0}^{n-1}(n+2) C^{\mathbf{e}_s-\mathbf{e}_t}.\]

(iii) If we take $n-2$ terms from $\left(1+\frac{dh}{w-z}\right) \prod_{l=1}^{n-1}\left(1+\frac{dh}{lz} \right)$, one $dh$ from $I_{n,\d}(z,w,h)$ and one $h$ from $(1-\frac{h}{w-z}+\ldots)^{n+2}\prod_{l=1}^{n-1}(1-\frac{h}{lz}+\ldots)^{n+2}$ then the contribution is 
\[C=\sum_{s=0}^{n-1}\sum_{t=0}^{n-1}\sum_{u=0}^{n-1}2\d n^3(n+2) C^{\mathbf{e}_s+\mathbf{e}_t-\mathbf{e}_u}.\]

(iv) Finally, if we take $n-2$ terms from $\left(1+\frac{dh}{w-z}\right) \prod_{l=1}^{n-1}\left(1+\frac{dh}{lz} \right)$, one $dh$ and one $h$ from $I_{n,\d}(z,w,h)$ then the contribution is
\[-D=-\sum_{s=0}^{n-1}\sum_{t=0}^{n-1}4\d n^3(n^2-1) C^{\mathbf{e}_s+\mathbf{e}_t}.\]

Using the positivity of the Taylor expansion, we obtain the following extension of \eqref{compare}, with the exact same proof using \eqref{cbi}.
\begin{lemma}\label{keylemma} \begin{enumerate}
\item For any integer partition $\bi=(i_1,\ldots, i_{n-1}$ with $0\le i_0+\ldots +i_{n-1}<n$ we have  
\[C^{\bi+\be_s}< \begin{cases} \frac{C^{\bi}}{2n \epsilon} & 1\le s \le n-1\\
C^{\bi} & s=0 \end{cases}\]
and hence 
\[\frac{C^{\bi}}{n}<\sum_{s=0}^{n-1} C^{\bi+\mathbf{e_s}} < C^{\bi}\left(1+\frac{1}{\epsilon}\right)\]
\item If $i_0+\ldots+ i_{n-1}=0$ and $p=1/2(|i_0|+\ldots +|i_{n-1}|)$ is the sum of the positive elements of $\bi$ then  
\[C^{\bi}< n^pC^{\mathbf{0}}\]
\end{enumerate}
\end{lemma}
Hence 
\begin{align*}
A <& 2\d n^3(n+2)\left(1+\frac{1}{\epsilon}\right) C^{\mathbf{0}} \\
B < & (n+2)n^3 C^{\mathbf{0}}\\
C < & 2\d n^3 \left(1+\frac{1}{\epsilon}\right) B \\
D<& 4n\left(1+\frac{1}{\epsilon}\right)A 
\end{align*}

 Moreover, if $\d =O(\frac{1}{n^3})$ holds, then the dominant contributions of $p_n$ are $B$ and $C$, and they give 
 \[|p_n|< 2\left(2+\frac{1}{\epsilon}\right) n^4 C^{\mathbf{0}}.\]
Similar computation shows that the dominant part in $p_{n-s}$ for $0\le s \le n$ are the terms corresponding to the choice when we take $n-s-2$ terms from $\left(1+\frac{dh}{w-z}\right) \prod_{l=1}^{n}\left(1+\frac{dh}{lz} \right)$ one $dh$ and $h^u$ from $I_{n,\d}(z,w,h)$ and $h^{s+1-u}$ from $(1-\frac{h}{w-z}+\ldots)^{n+2}\prod_{l=1}^{n-1}(1-\frac{h}{lz}+\ldots)^{n+2}$. This contribution is less than 
\begin{equation}\label{contribution}
2\d n^3(n+2)^{s+1-u}4^u{n^2-1 \choose u} \sum_{\alpha_1,\ldots, \alpha_{s+2}=0}^{n-1}\sum_{\beta_1,\ldots, \beta_{s+1-u}=0}^{n-1} C^{(\mathbf{e}_{\a_1}+\ldots +\mathbf{e}_{\a_{s+2}})-(\mathbf{e}_{\b_1}+\ldots +\mathbf{e}_{\b_{s+1-u}})}
\end{equation}
Note that $s+1-u=0$ is allowed, in this case the second sum is empty. Applying the second part of  part of Lemma \ref{keylemma} with $p=s+1-u$, then the first part $s+2-(s+1-u)=u+1$ times, we get 
\[\sum_{\alpha_1,\ldots, \alpha_{s+2}=0}^{n-1}\sum_{\beta_1,\ldots, \beta_{s+1-u}=0}^{n-1} C^{(\mathbf{e}_{\a_1}+\ldots +\mathbf{e}_{\a_{s+2}})-(\mathbf{e}_{\b_1}+\ldots +\mathbf{e}_{\b_{s+1-u}})}<\left(1+\frac{1}{\epsilon}\right)^{u+1}n^{2(s+1-u)}n^{s+1-u}C^{\mathbf{0}}.\]
Hence the sum \eqref{contribution} is less than 
\[\left(1+\frac{1}{\epsilon}\right)^{u+1}\frac{2\d n^3 4^u (n+2)^{s+1-u}n^{2u}n^{2(s+1-u)}n^{s+1-u}C^{\mathbf{0}}}{u!}<\left(1+\frac{1}{\epsilon}\right)^{u+1}\frac{2\d 4^u n^{4s+7-u}C^{\mathbf{0}}}{u!}.\]
and 
\[|p_{n-s}|<\left(1+\frac{1}{\epsilon}\right)^{s+1}\d 4^{s+2} n^{4s+8}C^{\mathbf{0}}\]
To finish the proof, we calibrate $\epsilon$ and $\delta$ to give the best bound. We first fix $\epsilon=1/4$, $\d=\frac{\epsilon}{4n^3}=\frac{1}{16n^3}$. With this choice we have:
\begin{itemize}
\item $p_{n+1}(n,\d)>C^{\mathbf{0}}\left(1-\frac{\d n^3}{\epsilon}\right)=\frac{3}{4}C^{\mathbf{0}}>0$; 
\item $|p_{n-s}|<\d 12^{s+2} n^{4s+7}C^{\mathbf{0}}=\frac{3}{4}(2n)^{4s+4}C^{\mathbf{0}}$.
\end{itemize}
The Fujiwara estimation of Lemma \ref{estimation} works with $D=(2n)^4$. Hence for $\d=\frac{1}{16n^3}$ and $d>2(2n)^4$ the integral $\int_{\calx_n^{\GIT}} I_{n,\d}>0$ is positive, hence it provides the existence of nonzero sections in Theorem \ref{germtoentire}. Then Theorem \ref{maintwo} applied with  
\[d>\max(2(2n)^4,\frac{5n+3}{\d}+n+2)=16n^3(5n+3)+n+2\]
finishes the proof of Theorem \ref{mainthmtwo}.
\bibliographystyle{abbrv}
\bibliography{BercziKirwanHyperbolicity.bib}

\end{document}